\documentclass[a4paper,11pt]{amsart}

\usepackage[left=2cm,top=3cm,bottom=2.5cm,right=2.5cm]{geometry}

\usepackage{amsthm}
\usepackage{amssymb}
\usepackage{amsmath}
\usepackage{mathtools}
\usepackage{pdfpages}
\usepackage{changes}
\usepackage{exscale,cite,color,amsopn}
\usepackage[colorlinks=true,pdftex,unicode=true,linktocpage,bookmarksopen,hypertexnames=true]{hyperref}
\usepackage{csquotes}
\usepackage{aliascnt}
\usepackage{colortbl}
\usepackage{enumitem}
\usepackage{verbatim}
\usepackage{hyperref}
\usepackage{tikz-cd}
\usepackage[capitalise]{cleveref}
\usepackage{bm}
\usepackage{array}

\usetikzlibrary{trees}

\renewcommand{\leq}{\leqslant}
\renewcommand{\geq}{\geqslant}
\DeclareMathOperator{\Id}{Id}

\DeclareMathOperator{\As}{As}
\DeclareMathOperator{\Ab}{Ab}
\DeclareMathOperator{\FAb}{FAb}

\DeclareMathOperator{\DD}{d}

\newcommand{\N}{\mathbb{N}}
\newcommand{\Z}{\mathbb{Z}}

\newcommand{\Sym}{\operatorname{S}}

\newcommand{\Q}{\mathbb{Q}}

\renewcommand{\Im}{\operatorname{Im}}
\newcommand{\w}{\operatorname{\omega}}
\renewcommand{\deg}{\operatorname{\delta}}

\newcommand{\qu}{\triangleright}
\newcommand{\T}{\operatorname{T}}
\newcommand{\D}{\operatorname{D}}
\newcommand{\R}{\operatorname{R}}
\renewcommand{\AA}{\operatorname{A}}
\newcommand{\GS}{\mathcal{G}}
\newcommand{\CBar}{\Bar{\mathcal{C}}}
\newcommand{\CGroup}{\mathcal{C}}
\newcommand{\SetConjClasses}{\mathbf{C}}
\newcommand{\smallG}{\GS}
\newcommand{\bigG}{\widetilde{\GS}}
\newcommand{\bigGplus}{\bigG^+}
\newcommand{\bigGfull}{\bigG^{+,\mathrm{f}}}
\newcommand{\lf}{l^\mathrm{f}}
\newcommand{\ppart}{\mathcal{P}}
\newcommand{\pp}{\boldsymbol{\mathrm{p}}}
\newcommand{\Full}{\mathrm{FTS}}
\newcommand{\FullR}{\mathrm{FRS}}
\newcommand{\kbar}{\Bar{k}}

\newcommand{\genrel}[2]{\left\langle \ #1 \ \vline \ #2 \ \right\rangle}

\makeatletter
\numberwithin{equation}{section}
\numberwithin{figure}{section}
\numberwithin{table}{section}
\newtheorem{thm}{Theorem}[section]
\newtheorem*{thm*}{Theorem}
\newtheorem{lem}[thm]{Lemma}
\newtheorem{cor}[thm]{Corollary}
\newtheorem{pro}[thm]{Proposition}

\theoremstyle{definition}
\newtheorem{defn}[thm]{Definition}
\newtheorem{nota}[thm]{Notation}

\newtheorem{rem}[thm]{Remark}
\newtheorem{exa}[thm]{Example}

\makeatother

\title{How structure groups and monoids grow}
\author{Carsten Dietzel, Edouard Feingesicht, Victoria Lebed}
\date{\today}



\address[Carsten Dietzel]{Normandie Univ, UNICAEN, CNRS, LMNO, 14000 Caen, France}
\email{carsten.dietzel@unicaen.fr}

\address[Edouard Feingesicht]{Normandie Univ, UNICAEN, CNRS, LMNO, 14000 Caen, France}
\email{edouard.feingesicht@unicaen.fr}

\address[Victoria Lebed]{Normandie Univ, UNICAEN, CNRS, LMNO, 14000 Caen, France}
\email{victoria.lebed@unicaen.fr}

\begin{document}

\begin{abstract}
The structure groups and monoids of set-theoretic solutions to the Yang--Baxter Equation can be regarded as deformations of free abelian groups resp. monoids. In this work, we obtain explicit formulae for the growth series of the structure groups and monoids of transposition and dihedral quandles, and of the structure groups of permutation quandles. These quandles provide important families of YBE solutions. The intricate nature of our formulae confirms that, while preserving many nice properties of free abelian groups, even the simplest structure groups and monoids are remarkably rich objects. We also establish some structural properties and easily computable normal forms for the monoids considered.
\end{abstract}

\maketitle

\section*{Introduction}

The free abelian group on a set $X$ admits the following standard presentation:
\[
\FAb(X) = \genrel{e_x : x \in X}{e_x e_y = e_y e_x : x,y \in X}.
\]
This presentation suggests the following natural deformation. Consider a map $r \colon X \times X \to X \times X$ sending $(x,y)$ to $(\lambda_x(y), \rho_y(x))$, and put
\begin{equation}\label{E:StrGroup}
    \As(X,r) = \genrel{e_x : x \in X}{e_x e_y = e_{\lambda_x(y)} e_{\rho_y(x)} : x,y \in X}.    
\end{equation}

This defines the \emph{structure group} (alternatively called the \emph{associated}, \emph{enveloping}, or \emph{adjoint group}) of $(X,r)$. Imposing appropriate conditions on $r$, one forces $\As(X,r)$ to preserve the desired properties of $\FAb(X)$. For example, if $X$ is finite and $r$ is a solution to the Yang--Baxter Equation 
\[(r \times \Id_X)(\Id_X \times r) (r \times \Id_X) = (\Id_X \times r)(r \times \Id_X)(\Id_X \times r),\] 
that is involutive, meaning that $r^2=\Id_{X \times X}$, and non-degenerate, that is, $\lambda_x$ and $\rho_x$ are bijections for all $x\in X$, then $\As(X,r)$ is Bieberbach  \cite{GIVdB} and Garside \cite{YBEGarside}.
In particular, this gave the first systematic way of constructing Garside groups going beyond the foundational example of spherical Artin--Tits groups.

In the present work, we go in an orthogonal direction, and show that even for very simple maps $r$, the group $\As(X,r)$ might significantly differ from $\FAb(X)$. The invariant that we chose to measure this difference is the growth series. Recall that the \emph{growth series} of a group $G$ with respect to a generating set $S$ is the formal power series defined as
\[\GS_{G,S}(t) = \sum_{n = 0}^{\infty} \#\left(\overline{S}^n \setminus \left( \bigcup_{i=0}^{n-1} \overline{S}^i \right) \right) t^n = \sum_{g \in G}t^{l_S(g)}, \]
where $\overline{S} = S \cup \{s^{-1} : s \in S \}$; $\overline{S}^k = \{s_1 s_2 \cdots s_k :  s_1, \ldots, s_k \in \overline{S} \}$ for $k\geq 1$ and $\overline{S}^0 = \{1_G\}$; and $l_S(g)$ is the minimal number of elements of $\overline{S}$ needed to produce (multiplicatively) $g$. The growth series measures how fast the words in the letters $s^{\pm 1}$, $s \in S$, cover the entire group $G$ when the word length is allowed to grow. The generating set is omitted in notations when clear from the context. Thus, for $\As(X,r)$, we will always take $S= \{e_x : x \in X\}$. For the free abelian group $\Z^d = \FAb(\{1,\ldots,d\})$ which is also the structure group of the solution $r(x,y) = (y,x)$ on the set $X = \{1, \ldots,d \}$, we have  
\begin{equation}\label{E:GrowthFAb}
    \GS_{\Z^d}(t) = \left( \frac{1+t}{1-t} \right)^d.
\end{equation}
The typical questions on growth series are their rationality and asymptotic behaviour. Explicit computation of the entire series is usually beyond reach, since it requires a fine understanding of how the given group works. Some notable exceptions are \cite{GrowthSurfaceGroups,GrowthGraphProduct,GrowthRacines,GrowthDyer}. The reader is directed to the introductory book \cite{HowGroupsGrow} for basics on the growth of groups\footnote{Note however a misprint in Example 7 from \cite{HowGroupsGrow} treating the case $G=\Z^d$. For $d=3$, the $n$th coefficient should be $a_3(n) = 4n^2+2$ instead of $4n^2+4n+2$.}.

The question of determining the growth series of structure groups was first explicitly raised in \cite{Vendramin_Survey}. To our knowledge, no concrete computations have been carried out so far. According to \cite{LebVenStrGroups}, the structure group of an invertible non-degenerate YBE solution is related to the structure group of its \emph{derived solution} of the ``only-left-deformed'' form
\begin{equation}\label{E:QuandleSol}
  r_{\qu}(x,y) = (x \qu y, x)  
\end{equation}
by a bijective length-preserving group $1$-cocycle. Thus, the two structure groups have identical growth series. This explains why solutions of the form  \eqref{E:QuandleSol} are particularly important. They can also be regarded as \emph{quandle} solutions, but since we do not use the quandle theory here, we will only mention connections to quandles as remarks for specialists. 

In this work, we provide computations for four important families of YBE solutions. Concretely, we take as $X$:
\begin{enumerate}
    \item the set $\Sym_d$ of all permutations of $d$ elements;
    \item the set $\T_d$ of all transpositions $(i,j)$ in $\Sym_d$;
    \item the set $\D_d$ of all symmetries of a regular $d$-gon;
    \item the set $\R_d$ of all reflections inside $\D_d$.
\end{enumerate}
As $r$, we take the map $r_{\qu}$, where $\qu$ is the conjugation operation:
\[x \qu y = xyx^{-1}.\]
For any group, or a union of conjugacy classes thereof, $r_{\qu}$ is a bijective non-degenerate YBE solution. Our four solutions, called, respectively, \emph{permutation, transposition, dihedral} and \emph{reflection solutions} in what follows, correspond to conjugation quandles (and their subquandles) of symmetric and dihedral groups. Note that the name \emph{dihedral quandle} historically refers to $(\R_d, \qu)$ rather than to the whole $(\D_d, \qu)$; we will not use this name to avoid confusion. The map $r_{\qu}$ will be omitted in notations when clear from the context; thus, $\As(X)$ stands for $\As(X,r_{\qu})$.

Our first results are the following formulae for the growth series of $\As(\T_d)$ and $\As(\R_d)$:
\begin{align*}
  \GS_{\As(\T_d)}(t) &=\prod_{k=2}^{d-1} (1 + kt) \cdot  \frac{1+t^2}{1-t} +  t \cdot \prod_{k=1}^{d-1} (1 + kt) \cdot \left(  \sum_{k=1}^{d-1} \frac{k}{1+kt} \right),\\ 
  \GS_{\As(\R_d)}(t) &=\frac{2dt}{1-t}+1+(d-1)t^2 \hspace{1cm} \text{ for odd } d.
\end{align*}
They are derived from the relation we establish between these growth series $\bigG$ and the growth series $\smallG$ for $\Sym_d$ with respect to the generating set $\T_d$ (respectively, $\D_d$ w.r.t. $\R_d$):
\begin{equation}\label{E:GrSeriesAsTn}
    \bigG(t) = 
            t \cdot \frac{\DD}{\DD \! t} \smallG(t) + \frac{1 + t^2}{1 - t^2}\smallG(t),
\end{equation}
and classical formulae for $\GS_{\Sym_d,\T_d}$ and $\GS_{\D_d,\R_d}$. Note that \cref{E:GrSeriesAsTn} actually holds for a large family of groups, including all Coxeter groups of ADE type (Theorem \ref{thm:relating_minor_and_major_growth_series}).

We further compute the growth series of the structure groups of entire permutation and odd rank dihedral solutions (actually, of a more general type of solutions, see \cref{thm:growth_series_by_defect_terms}):
    \begin{align*} 
        \GS_{\As(\Sym_d \text{ or } \D_d)}(t) = \gamma_d \cdot \left( \frac{1+t}{1-t} \right)^{c_d} - (1+t)^2 \cdot \left(\sum_{\kbar \in \Z^{c_d-1}} \left( \gamma_d - \#\left( \prod_{i=1}^{c_d-1} \mathcal{C}_i^{k_i} \right) \right) \cdot t^{|\kbar|}\right).
    \end{align*}
 Here $\gamma_d$ is the size of the commutator subgroup and $c_d$ the number of conjugacy classes of $\Sym_d$ or $\D_d$; 
 the $\mathcal{C}_i$ are the conjugacy classes different from $\{1\}$; $\kbar$ stands for a tuple of integers $(k_1, \ldots, k_{c_d-1})$; and $|\kbar| = \sum_{i=1}^{c_d-1}|k_i|$. For $\Sym_d$, the cumbersome correction term in the end of the formula is a polynomial except for $d=4$, and can be explicitly computed starting from the multiplication table for the conjugacy classes of $\Sym_d$. The products of conjugacy classes of $\Sym_d$ are an interesting subject in itself, see for instance \cite{SymConj} and references therein. In particular, asymptotically the growth behaviour of $\As(\Sym_d)$ is close to that of a free abelian group, except for $d=4$. For dihedral groups $\D_d$ with odd non-prime $d$, the correction term is not a polynomial. 

In parallel to the structure group, one can define the \emph{structure monoid} $\As^+(X,r)$ of $(X,r)$ using the same recipe as \eqref{E:StrGroup}, but this time regarding it as a monoid presentation. This object tends to better capture the nature of the map $r$; see for instance \cite{StrMon1,StrMon2} and references therein. We note here that the growth series of $\As^+(X,r)$ coincides with the Hilbert--Poincaré series of the associated \emph{structure algebra} $K(X,r)$ over a field $K$; see for instance \cite{StrMonAlg2,StrMonAlg} for a detailed study of the latter object.

In Theorem \ref{T:GSforStrMonTranspositions}, we explicitly compute the growth series of the structure monoids $\As^+(\T_d)$. The result possesses a surprisingly compact yet intricate presentation when these monoids are considered simultaneously for all $d$:
    \[
    \sum_{d \geq 0}\frac{1}{d!}\GS_{\As^+(\T_d)}(t)x^d \ = \ \exp\left(\frac{(1-tx)^{-t}-1-t^4x}{t^2(1-t^2)}\right).
    \]
This should be compared with the much simpler growth series of the free abelian monoid $\N_0^d$:
\begin{equation}\label{E:GrowthFAbMon}
    \GS_{\N_0^d}(t) = \left( \frac{1}{1-t} \right)^d, \hspace*{1cm} \sum_{d \geq 0}\frac{1}{d!}\GS_{\N_0^d}(t)x^d  = \exp\left( \frac{x}{1-t}\right).
\end{equation}
The key ingredient of our computation is a structural study of the monoid $\As^+(\T_d)$, interesting per se. We decompose $\As^+(\T_d)$ as a semilattice of subsemigroups, isomorphic to the partition semilattice of rank $d$. Each subsemigroup is a direct product of what we call the \emph{full transposition semigroups} $\Full_k$. The latter are then injected into the direct product $\Sym_k \times \N$, where the second component keeps track of the length. This is a semigroup version of the group injection $\As(\T_d) \hookrightarrow \Sym_d \times \Z$ from \cite{LebedConj}, which does not work for the entire monoid $\As^+(\T_d)$.

Finally, \cref{T:GrowthStrMonoidReflectionFinite} settles the case of the structure monoids of $\R_d$ both for odd and even $d$:
 \[\GS_{\As^+(\R_d)} (t)=1+ d \cdot \left( t + \left( \sum_{c|d } \frac{\varphi(c)}{c} \right) t^2 + \tau(d) \frac{t^3}{1-t}  + \tau\left(\frac{d}{2}\right) \frac{t^4}{2(1-t)^2} \right),\]
where $\tau$ is the number-of-divisors function, and $\varphi$ is Euler's totient function. At the heart of this computation are once again a semilattice decomposition of $\As^+(\R_d)$, a reduction of necessary computations to the extremal subsemigroups of this semilattice, and the injection of the latter into elementary semi-direct products with easily extractable length. This injection is realised using easily computable numeric invariants and simple normal forms, which surprisingly turn out to be more approachable for the reflections $\R_{\infty}$ inside infinite dihedral groups $\D_{\infty}$ (see \cref{S:MonoidsInfiniteReflection}). We then return to the finite case by reducing modulo $d$.

\medskip
\textbf{Acknowledgements.} The authors thank Eddy Godelle and Stéphane Launois for their contributions to the working group on the growth of structure groups, which gradually reduced to a working subgroup containing the three current authors. 
 
\section{How structure groups of transposition- and reflection-type solutions grow}\label{S:GroupsTransp}

In this section, we explicitly compute the growth series of the structure group of any transposition solution $\T_d$, and any reflection solution $\R_d$ with odd $d$. Our computations can be adapted to a more general class of solutions, presented below.

\begin{defn}
A \emph{$\CBar$-presentation} of a group is a presentation involving relations of two types only:  
\begin{itemize}
\item \emph{conjugation relations} $a \qu b = c$ for certain triples of generators where $a \qu b = aba^{-1}$;
\item \emph{power relations} of the form $a^p=1$, for certain generators $a$ and certain integers $p \geq 2$ (at most one relation per conjugacy class).
\end{itemize}    
A group admitting such a presentation is called a \emph{$\CBar$-group}.
\end{defn}

This definition goes back to \cite{LebedConj}; see there for examples of very different nature. The letter $\CGroup$ in the name stands for \emph{conjugation}, and the bar for the quotient by powers of generators, as in the finite quotients of certain structure groups.

We will work here with a particular type of $\CBar$-groups:

\begin{defn}
A \emph{$\CBar$-presentation of class $2$} is a $\CBar$-presentation whose set of generators $C$ is a single conjugacy class, and which has exactly one power relation, $a^2=1$. 
\end{defn}

Examples include:
\begin{enumerate}
    \item the symmetric group $\Sym_d$ generated by the set of transpositions $\T_d$, and, more generally, Coxeter groups of ADE type;
    \item the dihedral group $\D_d$ for odd $d$, generated by the set of reflections $\R_d$;
    \item finite quotients of the structure groups of finite connected involutory quandles \cite{LebVenStrGroups}.
\end{enumerate}    
Endow $C$ with the YBE solution $r_{\qu}$ from \eqref{E:QuandleSol}. We will first show how to express the growth series of the structure group $\As(C)$ (with respect to the standard generators $e_x$) in terms of the growth series of the original group $G$ with respect to the generating set $C$:

\begin{thm} \label{thm:relating_minor_and_major_growth_series}
    Let $G$ be a group with a $\CBar$-presentation of class $2$, with a set of generators $C$. Then one has the following relation between two group growth series:
    \begin{equation} \label{E:relating_minor_and_major_growth_series}
        \GS_{\As(C)}(t) = \frac{1 + t^2}{1 - t^2} \cdot \GS_{G,{C}}(t) +
            t \cdot \frac{\DD}{\DD \! t} \GS_{G,{C}}(t)
    \end{equation}    
\end{thm}

Before giving a proof, let us explore the consequences of this relation in our favourite examples. First, take $G=\D_d$ and $C=\R_d$ for odd $d$. The reflection lengths of symmetries of a regular $d$-gon look as follows:
\begin{enumerate}
    \item $0$ for the identity:
    \item $1$ for all $d$ reflections;
    \item $2$ for all $d-1$ non-trivial rotations.
\end{enumerate}
Thus
\[\GS_{\D_d,\R_d}(t) = 1 + dt + (d-1)t^2. \]
Plugging this into \eqref{E:relating_minor_and_major_growth_series}, we get
    \begin{align*} 
        \GS_{\As(\R_d)}(t) &= \frac{1 + t^2}{1 - t^2} \cdot (1 + dt + (d-1)t^2) +
            t \cdot (d + 2(d-1)t)\\
            &= \frac{1 + t^2}{1 - t} \cdot (1  + (d-1)t) +
            dt + 2(d-1)t^2\\
            &= \frac{(1 + t^2)\cdot (1  + (d-1)t) + (1-t) \cdot (-1 +
            dt + (d-1)t^2)}{1 - t}  +1+(d-1)t^2\\
            &= \frac{2dt}{1 - t}  +1+(d-1)t^2.
    \end{align*} 
    This formula can also be obtained by showing that every element of $\As(\R_d)$ can be uniquely presented in the form $e_0^ke_x$. This can be deduced from \cite{Clauwens,IglesiasVen}. 
    Results in \cite{LebedConj} can furthermore be applied to show that for even $d$ there is an explicit isomorphism $\As(\R_d) \cong \Z^2 \underset{\Z_2 \times \Z_2}{\times}\D_d$ which can be used to show that $\GS_{\As(\R_d)}(t) = \frac{d}{2}\left( \frac{1+t}{1-t} \right)^2 + \left( \frac{d}{2}-1 \right)(t^2 - 1)$, in that case.

    \medskip
    
Second, take $G=\Sym_d$, $C=\T_d$. To improve readability, we will use the notation 
\[\smallG_d =\GS_{\Sym_d, \T_d}, \hspace*{1cm} \bigG_d = \GS_{\As(\T_d)}\] 
for the growth series considered, which will be referred to as the \emph{small} and the \emph{big} series respectively. We will need the following formula for the small series:
\begin{pro}
    One has
    \begin{equation}\label{E:Solomon}
        \smallG_d(t) = \prod_{k=1}^{d-1} (1 + kt).
    \end{equation}
\end{pro}
This formula admits many generalisations, and can be deduced from the classical work on finite unitary reflection groups, see\cite{ShephardTodd}. We will present here an elementary inductive proof for completeness.

\begin{proof}
Since in our case all generators are involutive, in the formula
    \begin{equation*}
        \GS_{G,C}(t) = \sum_{g \in G} t^{l_C(g)}
    \end{equation*}
one can take as $l_C(g)$ is the minimal number of generators from $C$ needed to write $g$ as a product. The shortest expressions of a permutation in terms of transpositions correspond to cycle decompositions. Now, the $d+1$st element can be inserted into a permutation $\sigma \in \Sym_d$ either as a fixed point, which preserves the length of the permutation (in terms of transpositions), or after any of the $d$ elements in the cycle decomposition of $\sigma$, which increases the length of the permutation by $1$ and thus adds a factor $t$. Hence the recursion relation  
\begin{equation}\label{E:SolomonREcursion}
        \smallG_{d+1}(t) = (1 + dt) \cdot \smallG_d(t).
\end{equation}  
Together with $\smallG_1(t)=1$, it yields \eqref{E:Solomon}.
\end{proof}

Inserting this formula into \eqref{E:relating_minor_and_major_growth_series}, one gets
\begin{cor} For any integer $d \geq 2$, one has
    \[
    \bigG_d(t) = \prod_{k=2}^{d-1} (1 + kt) \cdot  \frac{1+t^2}{1-t} +  t \cdot \prod_{k=1}^{d-1} (1 + kt) \cdot \left(  \sum_{k=1}^{d-1} \frac{k}{1+kt} \right).
    \]
\end{cor}

For small values of $d$, this becomes
\begin{align*}
    \bigG_2(t) &= \frac{1+t}{1-t},\\
    \bigG_3(t) &= \frac{1+t}{1-t}\cdot (-2t^2+4t+1),\\
    \bigG_4(t) &= \frac{1+t}{1-t}\cdot (-12t^3+13t^2+10t+1),
\end{align*}
which for $d>2$ differs significantly from the free abelian case \eqref{E:GrowthFAb}.

To see how the big series $\bigG_d$ changes with $d$, we establish a recursive formula, very much resembling the formula \eqref{E:SolomonREcursion} for the small series:
\begin{cor} For any integer $d \geq 2$, one has
\begin{equation}\label{E:AsTnREcursion}
\bigG_{d+1}(t) = (1 + dt) \cdot \bigG_d(t)  + dt \cdot \smallG_d(t).
\end{equation}
\end{cor}

\begin{proof}
Deriving the recursive formula \eqref{E:SolomonREcursion} for $\smallG_d$, one gets
\[\frac{\DD}{\DD \! t}\smallG_{d+1}(t) = (1 + dt) \cdot \frac{\DD}{\DD \! t} \smallG_d(t) + d \cdot \smallG_d(t).\]
Thus
\begin{align*}
    \bigG_{d+1}(t) &=
            t \cdot \frac{\DD}{\DD \! t} \smallG_{d+1}(t) + \frac{1 + t^2}{1 - t^2} \cdot \smallG_{d+1}(t) \\
            &= t \cdot \left( (1 + dt) \cdot \frac{\DD}{\DD \! t} \smallG_d(t) + d \cdot \smallG_d(t) \right) + \frac{1 + t^2}{1 - t^2} \cdot (1 + dt) \cdot \smallG_d(t)\\
            &= (1 + dt) \cdot \left( t \cdot  \frac{\DD}{\DD \! t} \smallG_d(t) + \frac{1 + t^2}{1 - t^2}  \cdot \smallG_d(t)\right) + dt \cdot \smallG_d(t)\\
            &= (1 + dt) \cdot \bigG_d(t)  + dt \cdot \smallG_d(t). \qedhere
\end{align*}
\end{proof}

\begin{proof}[Proof of Theorem \ref{thm:relating_minor_and_major_growth_series}]

As was shown in \cite{LebedConj}, the assignment
\begin{align}
    \As(C) &\to G \times \Z,\label{E:AsIntoDirectProduct}\\
    e_x &\mapsto (x,1)\notag
\end{align}
uniquely extends to an injective group morphism, whose image is
\[G \underset{\Z_2}{\times} \Z = \{\ (g,m) \in G \times \Z \ : \ l_C(g)  \equiv m \pmod 2 \ \}.\]
Identifying elements of $\As(C)$ with their images in $G \times \Z$, one can interpret the desired growth series as follows:
\begin{equation}\label{E:TwoSeriesProof1}
   \GS_{\As(C)}(t) = \sum_{(g,m) \in G \underset{\Z_2}{\times} \Z}t^{l(g,m)}, 
\end{equation}
where $l(g,m)$ denotes the length with respect to the generators $e_x,\ x \in C$. Let us show that this length reduces to the following elementary formula:
\[l(g,m) = \max(l_C(g), |m|).\]
Indeed, to construct $(g,m)$ from the generators $(x,1),\ x \in C$, one needs on the one hand at least $l_C(g)$ of them (to obtain $g$ as the first factor), and on the other hand at least $|m|$ of them (to obtain $m$ as the second factor). So, $l(g,m) \geq \max(l_C(g), |m|)$. To prove the $\leq$ inequality, let $g=c_1 \cdots c_i$ be a minimal length representation of $g$ with respect to the generating set $C$. In the product $e_{c_1}^{\pm 1} \cdots e_{c_i}^{\pm 1}$, one can choose the signs $\pm 1$ to get as total degree any integer between $-i$ and $i$ having the same parity as $i$. If $l_C(g) \geq |m|$, this yields a representation of $(g,m)$ of length $l_C(g)$. In the opposite case $l_C(g) \leq |m|$, to get a representation of $(g,m)$ of length $|m|$, all signs $\pm 1$ should be the same as the sign of $m$, and several copies of the element $e_{c_1}^{\pm 2}$ (contributing only to the second factor of $G \times \Z$ since $c_1^{\pm 2}=1$ in $G$) should be added. Let us now split \eqref{E:TwoSeriesProof1} into two parts, according to which component yields the above $\max$. We will abbreviate $l_C(g)$ as $l(g)$ here.
\begin{align*}
    \GS_{\As(C)}(t) & \ =\  \sum_{g \in G} \sum_{\substack{m \in \Z, \\ l(g) \geq |m|,\\ m-l(g) \text{ even}}}t^{l(g)} \ +\ \sum_{g \in G}\sum_{\substack{m \in \Z, \\ l(g) < |m|,\\ m-l(g) \text{ even}}}t^{|m|}\\
    & \ =\ \sum_{g \in G}(l(g) + 1) t^{l(g)} \ +\ \sum_{g \in G}\sum_{\substack{m \in \Z, \\ |m|-l(g) > 0,\\ |m|-l(g) \text{ even}}}t^{|m|-l(g)}t^{l(g)}\\
    & \ =\ \sum_{g \in G}l(g) t^{l(g)} \ +\ \sum_{g \in G} t^{l(g)} \ +\ \sum_{g \in G}\sum_{i \in \N}2t^{2i}t^{l(g)}\\
    & \ =\ t \cdot \frac{\DD}{\DD \! t} \GS_{G,{C}}(t) \ +\ \sum_{g \in G}(1+\sum_{i \in \N}2t^{2i})t^{l(g)}
    \\
    & \ =\ t \cdot \frac{\DD}{\DD \! t} \GS_{G,{C}}(t) \ +\ \frac{1 + t^2}{1 - t^2} \cdot \GS_{G,{C}}(t). \qedhere 
\end{align*}
\end{proof}

\begin{rem}
Note that for a finite group $G$, in the analysis detailed in the proof, only the case $l(g) \leq |m|$ is possible for $|m|$ large enough. Thus 
\begin{align*}
    \GS_{\As(C)}(t) & \ =\ \sum_{m \in \Z} \frac{\#G}{2}t^{|m|} \ + \ \text{ some polynomial }\\
    & \ =\ \sum_{m \in \N} 2\frac{\#G}{2}t^{m} \ + \ \text{ some polynomial }\\
    & \ =\ \frac{\#G}{1-t} \ + \ \text{ some polynomial}.
    \end{align*}
    Thus all coefficients in the series, except for a finite number, equal  $\#G$.
\end{rem}

\section{How structure groups of permutation- and dihedral-type solutions grow}\label{S:GroupsPerm}

In this section, we explain how to compute the growth series of the structure group of the entire permutation solution $\Sym_d$ and dihedral solution $\D_d$. Again, we will work in a more general context, that of a finite $\CBar$-group $G$ with commutator length $1$. That is, each element in the commutator subgroup $[G,G]$ is a simple commutator $[a,b] = aba^{-1}b^{-1} = (a \qu b)b^{-1}$. Typical examples are symmetric groups (see \cite[Theorem 1]{Ore_Commutators}) and odd-order dihedral groups (for which the commutator subgroup is the subgroup of all rotations, which are always products of two reflections).

We will need the following notations. Given a group $G$, denote by $c$ the number of its conjugacy classes, and by $\SetConjClasses = \{ \mathcal{C}_i: 0 \leq i \leq c-1 \}$ the set of its conjugacy classes. We impose $\mathcal{C}_0 = \{ 1 \}$. Given a sequence of conjugacy classes $\mathcal{C}_{i_j}$ ($1 \leq j \leq n$), we define
\[
\prod_{j=1}^{n} \mathcal{C}_{i_j} = \left\{ g_1g_2 \ldots g_j \ : \ g_j \in \mathcal{C}_{i_j} (1 \leq j \leq n)  \right\}.
\]
An empty product, meaning that $n = 0$, is furthermore to be understood as $\{ 1 \}$. As each factor is invariant under conjugation, this product is independent of the specific order of factors. Furthermore, for some conjugacy class $\mathcal{C} = \mathcal{C}_i$ and an integer $k$, we define
\[
\mathcal{C}^k = \begin{cases}
    \prod_{j=1}^k \mathcal{C} & k \geq 0 \\
    \prod_{j=1}^{-k} \mathcal{C}^{-1} & k < 0,
\end{cases}
\]
where $\mathcal{C}^{-1} = \{ g^{-1}: g \in \mathcal{C} \}$.

\begin{defn}
The \emph{defect measure} of a finite group $G$ with a fixed order on its conjugacy classes is the map
\begin{align*}
    \delta_G \colon \Z^{c-1} &\longrightarrow \Z_{\geq 0},\\
\kbar &\longmapsto \#[G,G] - \# \left( \prod_{i=1}^{c-1} \mathcal{C}_i^{k_i}   \right).
\end{align*}  
Note that $\mathcal{C}_0$ is excluded from the above product. Furthermore, the \emph{defect series} of $G$ is defined as
\[
\Delta_G(t) = \sum_{\kbar \in \Z^{c-1}} \delta_G(\kbar) \cdot t^{|\kbar|},
\]
where $|\kbar| = \sum_{i=1}^{c-1}|k_i|$.
\end{defn}
Contrary to defect measures, the defect series is independent of the order on conjugacy classes of $G$. We remark here that $\delta_G$ takes only non-negative values as under the canonical homomorphism $\Ab: G \twoheadrightarrow G^{\mathrm{ab}}$, any product of conjugacy classes in $G$ is mapped to a singleton; therefore, the size of such a product cannot exceed $\#[G,G]$.

We now state the main result of this section:

\begin{thm} \label{thm:growth_series_by_defect_terms}
    Let $G$ be a finite $\CBar$-group with commutator length $1$, and with $c$ conjugacy classes. Then the growth series of the structure group of $(G, r_{\qu})$ can be computed as follows:
    \begin{equation} \label{eq:growth_series_with_defect}
        \GS_{\As(G)}(t) = \#[G,G] \cdot \left( \frac{1+t}{1-t} \right)^c - (1+t)^2 \cdot \Delta_G(t).
    \end{equation}
\end{thm}

In particular, the computation of the growth series of $\As(G)$ boils down to the computation of the defect series of $G$.

\begin{proof}
    Since the element $e_1$ generates a central subgroup of $\As(G)$ that is isomorphic to $\Z$, we can factor it out, and get
    \[
    \GS_{\As(G)}(t) = \GS_{(\Z,1)}(t) \cdot \GS_{\As(G_0)}(t) = \frac{1+t}{1-t} \cdot \GS_{\As(G_0)}(t),
    \]
    where $G_0 = G \setminus \{ 1 \}$. We will thus concentrate on $G_0$ in what follows.

    As was shown in \cite{LebedConj}, the assignment
\begin{align}
    \As(G_0) &\to G \times \Z^{c-1},\label{E:AsIntoDirectProductGeneral}\\
    e_x &\mapsto (x,1_i),\notag
\end{align}
where $1_i$ generates the component in $\Z^{c-1}$ corresponding to the conjugacy class $\mathcal{C}_i$ of $x$, uniquely extends to an injective group morphism. Its image is the pullback of the homomorphisms \[\Z^{c-1} \longrightarrow G^{\mathrm{ab}}=\Z_{p_1} \times \cdots \times \Z_{p_{c-1}} \overset{\Ab}{\longleftarrow} G_0.\]
Here $p_i$ is the power in the unique power relation $a^p=1$ with $a \in \mathcal{C}_i$ if there is one, and $p_i=+\infty$ otherwise. The map on the left is the canonical projection for each of the $c-1$ factors. It will be convenient for us to identify $\As(G_0)$ with its image in $G \times \Z^{c-1}$. 

Now, the abelianisation map $\Ab \colon \As(G_0) \twoheadrightarrow \Z^{c-1}$  is surjective. Its kernel is
\begin{equation}\label{E:KerAb}
  \Ab^{-1}(0) = [G,G] \times \{ 0 \} \subset G \times \Z^{c-1}.  
\end{equation}

We need to compute
  \[\GS_{\As(G_0)}(t) = \sum_{(g,\kbar) \in \As(G_0) }t^{l(g,\kbar)},\]
where $l(g,\kbar)$ denotes the length with respect to the generators $e_x,\ x \in G_0$. Let us show that this length is never too far from $|\kbar|$:
\[l(g,\kbar) \in \{|\kbar|,|\kbar|+2\}.\]
Since each $e_x$ contributes at most $1$ to exactly one factor in $\Z^{c-1}$, at least $|\kbar|$ such generators are needed to produce $(g,\kbar)$, hence $l(g,\kbar) \geq |\kbar|$. For the same reason, $l(g,\kbar)$ has the same parity as $|\kbar|$. If $l(g,\kbar)$ turns out to exceed $|\kbar|$, then pick any $h \in G$ with $l(h,\kbar) = |\kbar|$. As
\[(gh^{-1},0) = (g,\kbar)(h,\kbar)^{-1}  \in \Ab^{-1}(0) = [G,G] \times \{ 0 \},\]
we have $gh^{-1} \in [G,G]$. Our assumption on the commutator length implies that $gh^{-1}$ can be presented as $[a,b] =  (a \qu b)b^{-1}$. Since $a \qu b$ and $b$ lie in the same conjugacy class, say $\mathcal{C}_j$ (with $j \neq 0$), one can write 
\[(gh^{-1},0) = (a \qu b,1_j)(b,1_j)^{-1} = e_{a \qu b} e_b^{-1}.\]
Therefore,
\[(g,\kbar) = (gh^{-1},0) (h,\kbar) = (a \qu b,1_j)(b,1_j)^{-1}(h,\kbar)\]
can be written as a product of $|\kbar|+2$ generators and their inverses. Together with the parity remark above, this yields $l(g,\kbar) = |\kbar|+2$, as announced.

Now, the above analysis shows that the option $l(g,\kbar) = |\kbar|$ actually means that $g$ is realisable as a product of $|k_i|$ elements from each conjugacy class $\mathcal{C}_i$ (if $k_i>0$) or of inverses thereof (if $k_i<0$). This can be translated as
 $g \in \prod_{i=1}^{c-1}\mathcal{C}_i^{k_i}$. Thus, for a fixed $\kbar \in \Z^{c-1}$, the number of elements $(g,\kbar)$ realising this generic option is $\#\left( \prod_{i=1}^{c-1} \mathcal{C}_i^{k_i}   \right)$. Then the number of elements $(g,\kbar)$ realising the defect option $l(g,\kbar) = |\kbar|+2$ is \[ \#\Ab^{-1}(\kbar) - \# \left( \prod_{i=1}^{c-1} \mathcal{C}_i^{k_i}   \right) = \#[G,G]- \# \left( \prod_{i=1}^{c-1} \mathcal{C}_i^{k_i}   \right) = \delta_G(\kbar).\] 
 Indeed, \cref{E:KerAb} together with the surjectivity of $\Ab$ implies that all fibers of $\Ab$ have the same cardinality, $\#[G,G]$.

Summarising, we get
\begin{align*}
\GS_{\As(G_0)}(t) & = \sum_{\kbar \in \Z^{c-1}} ( \#[G,G]-\delta_G(\kbar) )t^{|k|} + \delta_G(\kbar)t^{|\kbar|+2} \\
        & = \#[G,G] \cdot \sum_{\kbar \in \Z^{c-1}} t^{|\kbar|} + (t^2-1) \cdot \sum_{\kbar \in \Z^{c-1}} \delta_G(\kbar) t^{|\kbar|} \\
        & = \#[G,G] \cdot \left( \frac{1+t}{1-t} \right)^{c-1} + (t^2-1) \Delta_G(t).
    \end{align*}
Multiplying by the missing factor $\GS_{(\Z,1)}(t) = \frac{1+t}{1-t}$, one recovers the desired formula for the entire growth series $\GS_{\As(G)}(t)$.
\end{proof}

To apply \cref{thm:growth_series_by_defect_terms} to our favourite groups $G=\Sym_d$, we need a thorough study of their defects. First, in this case each conjugacy class $\mathcal{C}$ satisfies $\mathcal{C}^{-1} = \mathcal{C}$, hence
\begin{equation}\label{E:signs}
\delta_{\Sym_d}(k_1,k_2,\ldots,k_{c-1}) = \delta_{\Sym_d}(|k_1|,|k_2|,\ldots,|k_{c-1}|),  
\end{equation}
i.e. the defect measure is independent of the signs of its arguments. This simplifies the computation of the defect series:
\begin{equation} \label{eq:defect_series_of_Sn}
    \Delta_{\Sym_d}(t) = \sum_{\kbar \in \Z_{\geq 0}^{c-1}} 2^{\#\{ i \in \{1, \ldots, c-1 \}: k_i \neq 0 \}} \cdot \delta_{\Sym_d}(\kbar) t^{|\kbar|}.
\end{equation}

We now demonstrate how to use this to compute $\GS_{\As(\Sym_d)}(t)$ for $d=3$ and $d=4$. 

First, let us list the conjugacy classes of $\Sym_3$:
\begin{center}
\begin{tabular}{c||c|c|c}
     $i$ & $0$ & $1$ & $2$  \\
     \hline $\mathcal{C}_i$  & $\{ 1 \}$ & ${}^{\Sym_3}(1\ 2)$ & ${}^{\Sym_3}(1\ 2\ 3)$ \\
     \hline $\#\mathcal{C}_i$ & $1$ & $3$ & $2$ 
\end{tabular}
\end{center}
Here $c=3$, and $[\Sym_3,\Sym_3]=\AA_3=\mathcal{C}_0 \cup \mathcal{C}_2$ is of size $3$. We need to determine all non-zero defects $\delta_{\Sym_3}(\kbar)$ for $k_i \geq 0$. To do this, it is essential to understand the products of all pairs of conjugacy classes, which are necessarily unions of conjugacy classes. We therefore build a multiplication table whose $i,j$-entry contains the indices of the conjugacy classes contained in $\mathcal{C}_i \cdot \mathcal{C}_j$:
\begin{center}
\begin{tabular}{c|| c| c| c}
     $\cdot$ & $0$ & $1$ & $2$  \\ \hline\hline
     $0$ & $0$ & $1$ & $2$  \\\hline
     $1$ & $1$ & $0,2$ & $1$ \\\hline
     $2$ & $2$ & $1$ & $0,2$ 
\end{tabular}
\end{center}
From this table, one reads that $\#\left( \prod_{i=1}^{2} \mathcal{C}_i^{k_i}   \right) = 3 = \#[\Sym_3,\Sym_3]$ whenever $|\kbar| = k_1+k_2 \geq 2$. That is, the defect measure $\delta_{\Sym_3}(\kbar)$ vanishes for such $\kbar$. For the remaining entries, we compute $\delta_{\Sym_3}$ as follows:
\begin{align*}
    \delta_{\Sym_3}(0,0)&=3-1=2,\\
    \delta_{\Sym_3}(1,0)&=3-\#\mathcal{C}_1=3-3=0,\\
    \delta_{\Sym_3}(0,1)&=3-\#\mathcal{C}_2=3-2=1.
\end{align*}
Thus 
\[\Delta_{\Sym_3}(t) =  2^0 \cdot \delta_{\Sym_3}(0,0) t^{0+0} + 2^1 \cdot \delta_{\Sym_3}(1,0) t^{1+0} + 2^1 \cdot \delta_{\Sym_3}(0,1) t^{0+1} =2+2t,\]
and
\begin{align*}
\GS_{\As(\Sym_3)}(t) &= 3 \cdot \left( \frac{1+t}{1-t} \right)^3 - (1+t)^2 \cdot (2+2t) = 3 \cdot \left( \frac{1+t}{1-t} \right)^3 - 2 \cdot (1+t)^3\\
&= (1+t)^3 \cdot \left( \frac{3}{\left(1-t \right)^3} - 2 \right) .    
\end{align*}

For $d=4$, the $c=5$ conjugacy classes are as follows:
\begin{center}
\begin{tabular}{c||c|c|c|c|c}
     $i$ & $0$ & $1$ & $2$ & $3$ & $4$ \\
     \hline $\mathcal{C}_i$  & $\{ 1 \}$ & ${}^{\Sym_4}(1\ 2)$ & ${}^{\Sym_4}(1\ 2)(3\ 4)$ & ${}^{\Sym_4}(1\ 2\ 3)$ & ${}^{\Sym_4}(1\ 2\ 3\ 4)$ \\
     \hline $\#\mathcal{C}_i$ & $1$ & $6$ & $3$ & $8$ & $6$     
\end{tabular}
\end{center}
The commutator subgroup $[\Sym_4,\Sym_4]=\AA_4=\mathcal{C}_0 \cup \mathcal{C}_2 \cup \mathcal{C}_3$ is of size $12$. The multiplication table for its conjugacy classes is
\begin{center}
\begin{tabular}{c| c c c c c}
     $\cdot$ & $0$ & $1$ & $2$ & $3$ & $4$ \\ \hline
     $0$ & $0$ & $1$ & $2$ & $3$ & $4$ \\
     $1$ & $1$ & $0,2,3$ & $1,4$ & $1,4$ & $2,3$ \\
     $2$ & $2$ & $1,4$ & $0,2$ & $3$ & $1,4$ \\
     $3$ & $3$ & $1,4$ & $3$ & $0,2,3$ & $1,4$ \\
     $4$ & $4$ & $2,3$ & $1,4$ & $1,4$ & $0,2,3$
\end{tabular}
\end{center}
In this table appears a phenomenon which makes $\Sym_4$ special among symmetric groups. We have $\mathcal{C}_2^m = \mathcal{C}_2 \cup \mathcal{C}_0$ (of size $4<12$) for all $m \geq 2$, and $\mathcal{C}_2^m \cdot \mathcal{C}_3 = \mathcal{C}_3$ (of size $8<12$) for all $m \geq 1$. In particular, the defect series $\Delta_{\Sym_4}$ is not a polynomial. Also, $\mathcal{C}_1 \cdot \mathcal{C}_4 = \mathcal{C}_2 \cup \mathcal{C}_3$ is of size $11$. Other multiple products of conjugacy classes yield either $\mathcal{C}_0 \cup \mathcal{C}_2 \cup \mathcal{C}_3$ or $\mathcal{C}_1 \cup \mathcal{C}_4$, and are of size $12 = \#[\Sym_4,\Sym_4]$. We summarise non-trivial defects in the following table:
\begin{center}
\begin{tabular}{c|c}
     $\kbar$ & $\delta_{\Sym_4}(\kbar)$ \\ \hline
     $(0,0,0,0)$ & $11$  \\
     $(1,0,0,0)$ &  $6$ \\
     $(0,1,0,0)$ & $9$  \\
     $(0,0,1,0)$ & $4$ \\
     $(0,0,0,1)$ & $6$ \\
     $(1,0,0,1)$ & $1$ \\
     $(0,m,0,0), m \geq 2$ & $8$ \\
     $(0,m,1,0), m \geq 1$ & $4$
\end{tabular}
\end{center}

Using \cref{eq:defect_series_of_Sn}, we can now compute:
\begin{equation*}
    \Delta_{\Sym_4}(t) = 11 + 2^1 \cdot (6+9+4+6)t + 2^2 \cdot t^2 + (2^1 \cdot 8 + 2^2 \cdot 4) \cdot  \sum_{m=2}^{\infty}t^m = 11 + 50t + 4t^2 + \frac{32t^2}{1-t},
\end{equation*}
hence
\begin{align*}
\GS_{\As(\Sym_4)}(t) &= 12 \cdot \left( \frac{1+t}{1-t} \right)^5 - (1+t)^2 \cdot \left(11 + 50t + 4t^2 + \frac{32t^2}{1-t}\right).
\end{align*}

Note that for $d \geq 3$ with $d \neq 4$, the simplicity of $\AA_d = [\Sym_d,\Sym_d]$ guarantees that for any nontrivial conjugacy class $\{ 1 \} \neq \mathcal{C}$ of $\Sym_d$, we have $\mathcal{C}^{2m} = \AA_d$ for $m$ large enough, hence $\delta_{\Sym_d}(\kbar) = 0$ for  sufficiently large $|\kbar|$. In particular, $\Delta_{\Sym_d}(t)$ is a \emph{polynomial} that can be computed explicitly.

The table below contains the numerators of $\GS_{\As(\Sym_d)}(t)$ (the denominators being $(1-t)^c$), obtained using a computer program\footnote{A link to the \texttt{GAP}-file is available under \texttt{https://sites.google.com/view/carstendietzel/} in the respective entry in the list of publications}: 

\medskip
\begin{tabular}{c|l|c} 
   $d$  & \hspace*{1cm} numerator of $\GS_{\As(\Sym_d)}(t)$  \\
   \hline
    $5$ & $4t^{12}+92t^{11}-46t^{10}-1455t^{9}+3505t^{8}-980t^{7}-4760t^{6}+7150t^{5}+2050t^{4}-1200t^{3}$\\&$+3086t^{2}+233t+1$ \\
    $6$ & $8t^{17}+640t^{16}+1836t^{15}-42306t^{14}+172997t^{13}-294051t^{12}+93174t^{11}+467324t^{10}$\\
& $-728893t^{9}+339031t^{8}+530288t^{7}-368178t^{6}+339579t^{5}+214699t^{4}-41778t^{3}$\\
& $+51480t^{2}+1429t+1$ \\
    $7$ & $4t^{22}+1500t^{21}-9052t^{20}+59768t^{19}-604186t^{18}+3616477t^{17}-12219475t^{16}+23927860t^{15}$\\
& $-22726364t^{14}-7980432t^{13}+55843840t^{12}-65892060t^{11}+35195992t^{10}+44917674t^{9}$\\
& $-38168278t^{8}+47903548t^{7}+11652920t^{6}-3556516t^{5}+11546372t^{4}-1720204t^{3}$\\
& $+775906t^{2}+10065t+1$ \\
    $8$ & $-4t^{30}-10144t^{29}+134812t^{28}-972046t^{27}+5044568t^{26}-14288042t^{25}-34866279t^{24}$\\
& $+606830872t^{23}-3360342688t^{22}+11590895976t^{21}-27738090666t^{20}+46699079226t^{19}$\\
& $-50843972868t^{18}+21390854702t^{17}+43714754431t^{16}-99390934656t^{15}+121989261436t^{14}$\\
& $-58116641248t^{13}+8695013472t^{12}+67561029126t^{11}-46119135216t^{10}+47606045586t^{9}$\\
& $-7722313273t^{8}+5034722152t^{7}+3123671976t^{6}-634489832t^{5}+555116750t^{4}-53617970t^{3}$\\
& $+14307868t^{2}+80618t+1$      
\end{tabular}

\medskip
Let us next turn to the dihedral group $G=\D_d$ for odd $d$. It is of size $2d$. Its conjugacy classes are $\mathcal{C}_0 = \{1\}$, $\mathcal{C}_1=\R_d$ (the set of $d$ reflections), and $\mathcal{C}_i=\{\rho^{i-1},\rho^{1-i}\}$ for $i=2,\ldots, \frac{d+1}{2}$. Here $\rho$ is the rotation by the angle $\frac{2\pi}{d}$. This yields $\frac{d+3}{2}$ classes in total. Omitting the reflection class $\mathcal{C}_1$, we get the subgroup $\operatorname{Rot}_d$ of all the $d$ rotation in $\D_d$, which is isomorphic to the cyclic group $\Z_d$. It turns out to be the derived subgroup of $\D_d$: $[\D_d,\D_d] = \operatorname{Rot}_d$. Relation $\mathcal{C}^{-1} = \mathcal{C}$ for all conjugacy classes allows us to compute the defect measures for positive entries only, and to use the relations \eqref{E:signs} and \eqref{eq:defect_series_of_Sn}.

We are now ready to compute $\GS_{\As(\D_5)}(t)$. (The case $d=3$ is of no interest since $\D_3=\Sym_3$.) The multiplication table for the conjugacy classes of $\D_5$ looks as follows:
\begin{center}
\begin{tabular}{c|| c| c| c| c}
     $\cdot$ & $0$ & $1$ & $2$ & $3$ \\ \hline\hline
     $0$ & $0$ & $1$     & $2$   & $3$ \\\hline
     $1$ & $1$ & $0,2,3$ & $1$   & $1$\\\hline
     $2$ & $2$ & $1$     & $0,3$ & $2,3$\\\hline
     $3$ & $3$ & $1$     & $2,3$ & $0,2$ 
\end{tabular}
\end{center}
We thus have $\#\left( \prod_{i=1}^{3} \mathcal{C}_i^{k_i}   \right) = 5 = \#[\D_5,\D_5]$ whenever $k_1 \geq 1$, or $k_2 \geq 4$, or $k_3 \geq 4$, or $k_2 + k_3 \geq 3$ and $k_2k_3 \neq 0$. For the remaining entries, we compute $\delta_{\Sym_3}$ as follows:
\begin{align*}
    \delta_{\D_5}(0,0,0)&=5-1=4,\\
    \delta_{\D_5}(0,1,0)&=5-\#\mathcal{C}_2=5-2=3,\\
    \delta_{\D_5}(0,0,1)&=5-\#\mathcal{C}_3=5-2=3,\\
    \delta_{\D_5}(0,1,1)&=5-\#(\mathcal{C}_2 \times \mathcal{C}_3)=5-4=1,\\
    \delta_{\D_5}(0,2,0)&=\delta_{\D_5}(0,0,2)=5-3=2,\\
    \delta_{\D_5}(0,3,0)&=\delta_{\D_5}(0,0,3)=5-4=1.\\
\end{align*}
Thus  
\begin{equation*}
    \Delta_{\D_5}(t) = 4 + 2^1 \cdot (3+3)t + 2^2 \cdot t^2 + 2^1 \cdot (2+2)t^2 + 2^1 \cdot (1+1)t^3 = 4+12t+12t^2+4t^3,
\end{equation*}
and
\begin{align*}
\GS_{\As(\D_5)}(t) &= 5 \cdot \left( \frac{1+t}{1-t} \right)^4 - (1+t)^2 \cdot \left( 4+12t+12t^2+4t^3\right) = 5 \cdot \left( \frac{1+t}{1-t} \right)^4 - 4(1+t)^5.
\end{align*}


For $d=7$, we get 
\begin{align*}
\GS_{\As(\D_7)}(t) &= 7 \cdot \left( \frac{1+t}{1-t} \right)^5 - 6(1+t)^7 + 6t^3(1+t)^3.
\end{align*}

For $d=9$
\begin{equation*}
    \Delta_{\D_9}(t) =8(1+7t+21t^2+35t^3)-10t^3+ \text{ higher terms}.
\end{equation*}

Our last observation is that the defect series $\Delta_{\D_d}(t)$ does not need to be a polynomial in general. For instance for $d=9$, the classes $\mathcal{C}_0 = \{1\}$  and $\mathcal{C}_4=\{\rho^{3},\rho^{-3}\}$ generate a subgroup of $3$ rotations inside $\D_9$ (a copy of $\Z_3$ inside $\operatorname{Rot}_9 \cong \Z_9$), hence $\delta_{\D_9}(0,0,0,k,0) = 9-3 = 6 \neq 0$ for all $k \geq 2$.

\section{How structure monoids of transposition solutions grow}\label{S:MonoidsTransp}

Let us now turn to the growth series of the structure monoid $\As^+(\T_d)$ of the transposition solution. The basic ingredient of our computations in the group case was the injection \eqref{E:AsIntoDirectProduct} of $\As(\T_d)$ into a more accessible group, the direct product $\Sym_d \times \Z$. The corresponding map is no longer injective for the monoid. Indeed, the elements $e_x^2$, $x \in \T_d$, are all \emph{frozen} in the monoid $\As^+(\T_d)$ (since no non-trivial braiding moves apply to them) and thus pairwise distinct, whereas the map \eqref{E:AsIntoDirectProduct} sends them all to $(\Id, 2) \in \Sym_d \times \Z$. However, the situation changes in the presence of other generators.  Concretely, for all $x,y \in \T_d$ we have 
\begin{equation}\label{E:e2Central}
    e_x^2e_y = e_xe_{x \qu y}e_x = e_{x \qu (x \qu y)}e_xe_x = e_{x^2 \qu y}e_x^2 = e_ye_x^2,
\end{equation}
hence the squares of all generators are central in $\As^+(\T_d)$. On the other hand,
\begin{equation}\label{E:e2Equal}
    e_ye_x^2 = e_{y \qu x}e_ye_x = e_{y \qu x}e_{y \qu x}e_y = e_{y \qu x}^2 e_y.
\end{equation}
Since $\T_d$ is a single conjugacy class in $\Sym_d$, relations \eqref{E:e2Central}-\eqref{E:e2Equal} imply that the squares $e_x^2$ are all equal in the group $\As(\T_d)$, whereas in the monoid $\As^+(\T_d)$ they are equal only in the presence of appropriate generators. To make this statement precise, we need the following definition.

\begin{defn}
    Let $w$ be a word in the letters $e_x$, $x \in \T_d$. The (undirected) \emph{graph} $\Gamma_w$ of $w$ is constructed on the vertices $1,2, \ldots, d$ by connecting $i$ and $j$ whenever the letter $e_{(i,j)}$ appears in $w$ at least once. Here, for each transposition we allow two notations, $(i,j)$ and $(j,i)$. The connected components of $\Gamma_w$ yield a partition of the set $\{1,2, \ldots, d\}$, called the \emph{partition} of $w$. A component is considered to be \emph{trivial} if it is a singleton.
\end{defn}

\begin{lem}
    Two words in the letters $e_x$, $x \in \T_d$, representing the same element of the monoid $\As^+(\T_d)$, have the same partition.
\end{lem}

\begin{proof}
    Indeed, the defining relations $e_{(i,j)} e_{(k,l)}= e_{(k,l)}e_{(i,j)}$ and $e_{(i,j)} e_{(k,j)}= e_{(k,i)}e_{(i,j)}$ do not change the connected components of the graph from the above definition. Here the indices $i,j,k,l$ are all distinct.
\end{proof}

Thus, one can talk about the partition $\pp(a) \in \ppart_d$ (the set of all partitions of the set $\{1,2, \ldots, d\}$) of an element $a \in \As^+(\T_d)$. This yields a map 
\[\pp \colon \As^+(\T_d) \to \ppart_d.\]

Recall that the \emph{join} $\pi \vee \rho$ of two partitions $\pi$ and $\rho$ is obtained by gluing together all intersecting parts from $\pi$ and $\rho$. Note that $\pi \vee \rho$ is the finest partition that is coarser than $\pi$ and $\rho$, therefore the join describes the semilattice structure of $\ppart_d$ with respect to the refinement relation. It is called the \emph{partition semilattice of rank $d$}.

\begin{pro}\label{P:Lattice}
    The monoid $\As^+(\T_d)$ splits as a disjoint union of its subsemigroups $\pp^{-1}(\pi)$, $\pi \in \ppart_d$. Moreover, for any elements $a \in \pp^{-1}(\pi)$ and $b \in \pp^{-1}(\rho)$, their product $ab$ lies in $\pp^{-1}(\pi \vee \rho)$.
\end{pro}

\begin{proof}
    By definition, to construct the graph of $ab$ on the vertices  $1,2, \ldots, d$, one takes as edges precisely the edges corresponding to the letters $e_{(i,j)}$ from a word $w$ representing $a$ (hence present in the graph of $w$), and the letters $e_{(i,j)}$ from a word $v$ representing $b$ (hence present  in the graph of $v$). Hence the connected components of the word $wv$ representing $ab$ are obtained by gluing together all intersecting components of the graphs of $w$ and $v$.
\end{proof}

The monoid $\As^+(\T_d)$ thus splits as a semilattice of semigroups $\pp^{-1}(\pi)$. To understand the latter, it suffices to study only the semigroups of the following particular type. 

\begin{defn}
    An element $a \in \As^+(\T_d)$ is called \emph{full} if the induced partition $\pp(a)$ has only one part, the whole set $\{1,2, \ldots, d\}$. The subsemigroup of  elements of $\As^+(\T_d)$ is called the \emph{full transposition semigroup} on $d$ elements, denoted by $\Full_d$.
\end{defn}
\begin{exa}
In $\As^+(\T_5)$, consider the elements 
\[a=e_{(1,2)}e_{(2,3)}e_{(4,5)}e_{(1,2)} \hspace*{1cm}\text{ and }\hspace*{1cm} b=e_{(2,3)}e_{(1,5)}e_{(2,5)}e_{(3,4)}e_{(1,2)}.\]
Their respective graphs are 
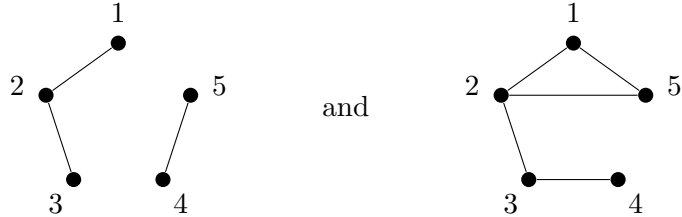
\[
 \vcenter{\hbox{ \begin{tikzpicture}[bullet/.style={circle, fill, inner sep=2pt}]
    \foreach \lab [count=\c, evaluate=\c as \ang using {18+72*\c}] in {1,2,3,4,5} {
       \node[bullet] (\c) at (\ang:10mm) {};
       \node at (\ang:14mm){$\lab$};
    }
    \draw(1)--(2)--(3);
    \draw(4)--(5);
  \end{tikzpicture}}} \hspace*{1cm} \text{ and }\hspace*{1cm}
    \vcenter{\hbox{\begin{tikzpicture}[bullet/.style={circle, fill, inner sep=2pt}]
    \foreach \lab [count=\c, evaluate=\c as \ang using {18+72*\c}] in {1,2,3,4,5} {
       \node[bullet] (\c) at (\ang:10mm) {};
       \node at (\ang:14mm){$\lab$};
    }
    \draw (1)--(5)--(2)--(1);
    \draw (2)--(3)--(4);
  \end{tikzpicture}}}\] 
Their respective partitions are $\pp(a)=\{\{1,2,3\},\{4,5\}\}$ and $\pp(b)=\{\{1,2,3,4,5\}\}$. So $b$ is full whereas $a$ is not.
\end{exa}

\begin{lem}\label{L:DecomposeInFull}
    Let $\pi \in \ppart_d$ be a partition with parts of size $i_1, \ldots,i_k$. Then the semigroup $\pp^{-1}(\pi)$ is isomorphic to the direct product of the following full transposition semigroups:
    \[\pp^{-1}(\pi) \cong \Full_{i_1} \times \cdots \times \Full_{i_k}.\]
    Note that a part of size $1$ gives rise to a trivial factor isomorphic to $\Full_1 = \{ 1 \}$.
\end{lem}

\begin{proof}
    The defining relations of $\As^+(\T_d)$ guarantee that two elements commute if their non-trivial components are mutually disjoint. This yields a surjection $\Full_{i_1} \times \cdots \times \Full_{i_k} \twoheadrightarrow  \pp^{-1}(\pi)$. It is injective since the defining relations of $\As^+(\T_d)$ applicable to a word representing an element from $\pp^{-1}(\pi)$ are either commutation relations for letters coming from different components $\Full_{i_j}$, or commutation or commutation-conjugation relations between letters from the same component of size ${i}$, which are valid in $\Full_{i}$.
\end{proof}

For small values of $d$, we have:
\begin{enumerate}
    \item $\As^+(\T_1) = \Full_1 = \{ 1 \}$, the trivial monoid. Note that $\T_1 = \{ \}$.
    \item $\As^+(\T_2) = \{1\} \sqcup \Full_{2}$, where $\Full_{2} \cong \N$ is generated by $e_{(1,2)}$.
    \item $\As^+(\T_3)$ contains a copy of $\{1\}$; three copies of $\Full_{2}$, generated by the three elements $e_{(i,j)}$, respectively, and one copy of $\Full_{3}$. The product of two elements from different copies of $\Full_{2}$ lands in $\Full_{3}$.
\end{enumerate}

The computation of the growth series of $\As^+(\T_d)$ is thus reduced to that for all full components. The nice thing about the semigroup $\Full_d$ is that it does inject into $\Sym_d \times \N$, as suggested by the computations at the beginning of this section. This opens the way to the techniques used in Section \ref{S:GroupsTransp}.

\begin{thm}\label{T:FullInjects}
    For $d \geq 2$, the assignment
\begin{align*}
    \As^+(\T_d) &\to \Sym_d \times \N,\\
    e_x &\mapsto (x,1)
\end{align*}
uniquely extends to a semigroup morphism. Its restriction to $\Full_d$, denoted by
\[\iota_d \colon \Full_d\to \Sym_d \times \N\]
is injective, and its image is
\[\Im(\iota_d) =  \{\ (g,2(d-1)-l(g)+2k) \in \Sym_d \times \N  : k \in \Z, k \geq 0 \ \}.\]
As usual, $l(g)$ is the minimal number of transpositions from $\T_d$ whose product yields $g$. 
\end{thm}

In the first non-trivial case $d =3$, the theorem identifies $\Full_3$ with the subsemigroup of $\Sym_3 \underset{\Z_2}{\times} \N$ obtained by discarding the elements $(\Id,2)$ and $((i,j),1)$.

\begin{proof}
  The computation
  \[(x \qu y,1)(x,1) = ((x \qu y) \cdot x,2) = (xy,2) = (x,1) (y,1)\]
  guarantees that we indeed have a well-defined semigroup morphism. 

  Our proof of injectivity is based on the following classical properties of the symmetric group $\Sym_d$ \cite[Proposition 1.6.1]{Bessis_dual_braid_monoid}:
  \begin{enumerate}
      \item\label{i:Sym1} The minimal number $l(g)$ of transpositions $(i,j) \in \T_d$ needed to write a permutation $g \in \Sym_d$ equals $d-c(g)$, where $c(g)$ is the number of cycles in $g$. 
      \item\label{i:Sym2}  Two minimal length representations of a $g \in \Sym_d$ can be related by a finite series of rewriting rules
      \begin{equation}\label{E:RewritingSd}
          (i,j) \cdot (k,l) \ \longleftrightarrow \ \left( (i,j) \qu (k,l)\right) \cdot (i,j)
      \end{equation} (going in any direction).
      \item\label{i:Sym3}   Any representation of a $g \in \Sym_d$ can be turned into one of minimal length using the rewriting rules \eqref{E:RewritingSd} and
      \begin{equation}\label{E:RewritingSd2}
          (i,j)^2 \ \longrightarrow \ 1.
      \end{equation}
  \end{enumerate}
  
  Now, take $a,b \in \Full_d$ represented by words $w$ and $v$ (in the generators $e_{(i,j)}$) respectively, and assume that $\iota_d(a)=\iota_d(b)=(g,m)$. Replacing each $e_{(i,j)}$ with $(i,j)$, one gets two representations of the same permutation $g \in \Sym_d$. Rewrite them to get minimal length representations, as explained in item \ref{i:Sym3} above. These rewriting procedures can be mimicked for the words $w$ and $v$. Indeed, the rewriting steps \eqref{E:RewritingSd} become $e_{(i,j)} e_{(k,l)} \leftrightarrow e_{(i,j) \qu (k,l)} e_{(i,j)}$, while \eqref{E:RewritingSd2} can be replaced with pushing the central element $e_{(i,j)}^2$ to the left of the word. For the two minimal length representations of $g$ obtained this way, mimic the rewriting procedure from item \ref{i:Sym2} in the same manner. This yields two words $w'=e_{(i_1,j_1)}^2 \cdots e_{(i_k,j_k)}^2u$ and $v'=e_{(s_1,t_1)}^2 \cdots e_{(s_k,t_k)}^2u$ in the generators $e_{(i,j)}$ representing $a$ and $b$ respectively. Here the common subword $u$ lifts one of the minimal length representatives of $g$. Note that the two words $w'$ and $v'$ have the same number $k$ of square factors $e_{(\bullet,\bullet)}^2$ since they have the same length $m$. Now, choose arbitrary representatives $r_1, \ldots, r_q$ of the parts of the partition $\pi$ of $u$. \cref{E:e2Central} allows one to move any factor $e_{(i,j)}^2$ in any of the two words towards any letter $e_{(s,t)}$ in $u$; then, with \cref{E:e2Equal}  one can make the transposition $(s,t)$ act on $(i,j)$ by conjugation; finally, the modified square can be returned to its initial position using \eqref{E:e2Central} again. This sequence of operations will be called a \emph{contamination} of $e_{(i,j)}^2$ in this proof. Repeating contamination for all the squares $e_{(i,j)}^2$ in the subwords of $w'$ and $v'$, one can replace $i$ and $j$ with our selected indices $r_1, \ldots, r_q$, whenever $i$ and $j$ come from distinct parts of $\pi$. Then, for the squares $e_{(i,j)}^2$ with $i$ and $j$ from the same part, $i$ can be first contaminated into the index $r_x$ from the same part, then, if there are at least two parts, to another $r_y$ using the remaining squares $e_{(s,t)}^2$ (recall that $a$ and $b$ are full), after which $j$ can be contaminated into $r_x$. Thus all the indices in the squares can be forced to be of the form $r_x$. To relate the words $w''$ and $v''$ obtained this way, one can use
  \begin{lem}\label{L:HandlingSquares}
     Consider a full element $a \in  \Full_d$ which can be represented as a product of squares of generators,  
      $a=e_{(i_1,j_1)}^2 \cdots e_{(i_k,j_k)}^2$. Then it can be rewritten in the following canonical way:
      \[a=e_{(1,2)}^2e_{(2,3)}^2 \cdots e_{(d-2,d-1)}^2e_{(d-1,d)}^{2p},\]
      for $p = k-(d-2)$.
  \end{lem}
  \begin{proof}
      We will use induction on $k$. An element $e_{(i_1,j_1)}^2$ can be full only if $d=2$, in which case the statement is tautological.

      Next, assume the statement established for a certain $k$, and consider a full element \[a=e_{(i_1,j_1)}^2 \cdots e_{(i_k,j_k)}^2e_{(i_{k+1},j_{k+1})}^2.\] 
      
      \textbf{Case 1.} The word remains full after omitting some square, say $e_{(i_{k+1},j_{k+1})}^2$ (commuting squares can always be rearranged in this way). The induction hypothesis then applies to the prefix $e_{(i_1,j_1)}^2 \cdots e_{(i_k,j_k)}^2$, allowing us to rewrite $a$ as 
      \[ a= e_{(1,2)}^2e_{(2,3)}^2 \cdots e_{(d-2,d-1)}^2e_{(d-1,d)}^{2p}e_{(i_{k+1},j_{k+1})}^2.\]
      The last square $e_{(i_{k+1},j_{k+1})}^2$ can then be contaminated into $e_{(d-1,d)}^2$. 
      
      \textbf{Case 2.} If on the contrary the removal of any edge disconnects the graph of $a$, then this graph is necessarily a tree. Pick any square $e_{(c,d)}^2$ where one of the indices is $d$. Using contamination by $e_{(c,d)}$, replace all other occurrences of the index $d$ with $c$. Finally, contaminate $e_{(c,d)}^2$ into $e_{(d-1,d)}^2$ (recall that the graph is connected, since $a$ is full). Then one can rewrite $a=be_{(d-1,d)}^2$, where $b \in \Full_{d-1}$ is a full product of $d-2$ squares on $d-1$ letters. The induction hypothesis allows us to rewrite it as $b=e_{(1,2)}^2e_{(2,3)}^2 \cdots e_{(d-2,d-1)}^2$, and we are done.
  \end{proof}
  The one-partition case is not covered by the above arguments. In this situation, all squares can be contaminated into the same square, say $e_{(1,2)}^2$, and the numbers of such squares in the two words are equal because of the second component of the map $\iota_d$.

  To compute the image of $\iota_d$, we need to find the minimal length $\lf(g)$ of a full representation of a permutation $g \in \Sym_d$ in terms of the transpositions $(i,j) \in \T_d$. Here \emph{full representations} are defined in the same way as the full elements of $\As^+(\T_d)$. Indeed, since in the image $\iota_d(a)=(g,m)$ the permutation $g$ and the integer $m$ are of the same parity, and since $m$ can be increased by $2$ by multiplying $a$ by $e_{(1,2)}^2$, we have 
  \[\Im(\iota_d)=\{\  (g,\lf(g)+2k) \in \Sym_d \times \N  : k \in \Z, k \geq 0 \ \}.\]  
  
  Starting with a representation of $g$ as a product of $l(g)$ transpositions, whose graph has $c(g)$ connected components, one can make it full by adding $c(g)-1$ squares ${(i,j)}^2$. Thus, 
  \[\lf(g) \leq l(g)+2(c(g)-1).\]
  
  On the other hand, any full representation of $g \in \Sym_d$ yields a full element $a \in \Full_d$ with $\iota_d(a) = (g,m)$. As explained before Lemma \ref{L:HandlingSquares}, $a$ can be rewritten as $e_{(i_1,j_1)}^2 \cdots e_{(i_k,j_k)}^2u$, where $u$ lifts a minimal length representation of $g$. Then $u$ is of length $l(g)$, and its graph has $c(g)$ connected components. To connect the graph, one needs at least $c(g)-1$ edges, hence $k \geq c(g)-1$. Thus,
  \[\lf(g) \geq l(g)+2(c(g)-1).\] 
  Hence the announced equality
  \[\lf(g) = l(g)+2(c(g)-1) = l(g)+2(d-l(g)-1) = 2(d-1)-l(g).\qedhere\]
\end{proof}

Note that the proof actually works for any group with a $\CBar$-presentation of class $2$ satisfying the rewriting conditions \ref{i:Sym2}-\ref{i:Sym3}.

The theorem gives for free the following useful property:

\begin{cor}
    The tautological map
    \begin{align*}
        \Full_d &\to \As(\T_d),\\
        e_{a_1}\ldots e_{a_k} &\mapsto e_{a_1}\ldots e_{a_k}
    \end{align*}
    is an injective morphism of semigroups.
\end{cor}

\begin{proof}
    Realising $\Full_d$ inside $\Sym_d \times \N$ using the map $\iota_d$, and $\As(\T_d)$ inside $\Sym_d \times \Z$ using \eqref{E:AsIntoDirectProduct}, one interprets the map from the statement as the standard inclusion $\Sym_d \times \N \hookrightarrow \Sym_d \times \Z$. 
\end{proof}

We have prepared all ingredients needed to compute the growth series of the structure monoids $\As^+(\T_d)$, which we organise into a single two-variable series. 

\begin{thm}\label{T:GSforStrMonTranspositions}
    The growth series $\bigGplus_d(t)$ of $\As^+(\T_d)$ can be computed by the following formula:
    \begin{equation*}
        \sum_{d \geq 0}\frac{1}{d!}\bigGplus_d(t)x^d \ = \ \exp\left(\frac{(1-tx)^{-t}-1-t^4x}{t^2(1-t^2)}\right).
    \end{equation*}
\end{thm}

Here we use the {exponential} $\exp$ defined for formal series $f = \sum_{k,l=0}^{\infty} r_{k,l} t^k x^l \in \Q[[t,x]]$ with $r_{0,0} = 0$ by
\[
\exp(f) = \sum_{n=0}^{\infty} \frac{f^n}{n!}.
\]
Note that, as usual, $\exp(f+g) = \exp(f)\exp(g)$.

Also, the expression $(1-tx)^{-t}$ should be understood as the formal binomial series 
\[\begin{aligned}
(1+y)^{\alpha }&=\sum _{n=0}^{\infty }\!{\binom {\alpha }{n}}y^{n}
=1+\alpha y+{\frac {\alpha (\alpha -1)}{2!}}y^{2}+{\frac {\alpha (\alpha -1)(\alpha -2)}{3!}}y^{3}+\cdots. \end{aligned}\]
with $\alpha = -t$ and $y = -tx$.

\begin{proof}
Let us first compute the restricted growth series $\bigGfull_d(t) = \sum_{g \in \Full_d}t^{l(g)}$ for $d \geq 2$. Theorem \ref{T:FullInjects} identifies the elements of $\Full_d$ with the couples $(g,2(d-1)-l(g)+2k) \in \Sym_d \times \N$, with $k \in \Z$, $k \geq 0$. The argument at the end of the proof of that theorem guarantees that the length of such an element, in terms of the generators $e_{(i,j)}$, is simply the second component $2(d-1)-l(g)+2k$. Thus
\begin{align*}
    \bigGfull_d(t) &= \sum_{g \in \Sym_d} \sum_{k \geq 0} t^{2(d-1)-l(g)+2k} = \frac{t^{2(d-1)}}{1-t^2}  \sum_{g \in \Sym_d}t^{-l(g)} \\
    &= \frac{t^{2(d-1)}}{1-t^2}\smallG_d(t^{-1}) =\frac{t^{2(d-1)}}{1-t^2}    
    \prod_{k=1}^{d-1} (1 + kt^{-1}) \\
    &=\frac{t^{d-2}}{1-t^2}    
    \prod_{k=0}^{d-1} (t + k) = \frac{(-t)^{d}}{t^2(1-t^2)} 
    \prod_{k=0}^{d-1} (-t - k) \\
    &=d!\frac{(-t)^{d}}{t^2(1-t^2)}  \binom {-t}{d}.
\end{align*}
We used the formula \eqref{E:Solomon} for the growth series $\smallG_d =\GS_{\Sym_d, \T_d}$. For $d=1$, the above formula fails. For future computations, we need its corrected version:
\begin{align*}
    \bigGfull_1(t) &\ =\ 1 \ = \ 1!\frac{(-t)^{1}}{t^2(1-t^2)}  \binom {-t}{1} \cdot (1-t^2).
\end{align*}
So, the correction factor is $(1-t^2)$.

To compute the growth series $\bigGplus_d(t)$ of the entire monoid $\As^+(\T_d)$, according to Proposition \ref{P:Lattice}, one needs to sum over all partitions $\pi \in \ppart_d$ of the set $\{1,2,\ldots,d\}$ the restricted growth series of the subsemigroups $\pp^{-1}(\pi)$. According to Lemma \ref{L:DecomposeInFull}, for a partition with parts of size $i_1, \ldots,i_k$, the latter restricted growth series is the product of the restricted growth series of the corresponding full transposition semigroups:
\[\GS_{\pp^{-1}(\pi)}(t) = \sum_{g \in \pp^{-1}(\pi)}t^{l(g)} = \bigGfull_{i_1}(t) \cdots \bigGfull_{i_k}(t).\] 
In particular, it depends only on the part sizes $i_1, \ldots,i_k$ of $\pi$. Now, take a partition 
\[\lambda = (n_1 \leq n_2 \leq \ldots \leq n_k) \vdash d\]
of an integer $d$. Denote by $\mu_i(\lambda)$ the number of its parts of size $i$. One has $\sum_i\mu_i(\lambda) = k$, and $\sum_i i\mu_i(\lambda) = \sum_j n_j = d$. In our computations, we will need the number of set partitions $\pi \in \ppart_d$ whose part sizes form the integer partition $\lambda$. It is given by the formula
\[\binom{d}{n_1,\ldots,n_k}\prod_i\frac{1}{\mu_i(\lambda)!}.\]
The preceding argument then yields:
\begin{align*}
\bigGplus_d(t) &= \sum_{\lambda= (n_1, \ldots, n_k) \vdash d}  \binom{d}{n_1,\ldots,n_k}\prod_i\frac{1}{\mu_i(\lambda)!} \prod_i  \left(\bigGfull_{i}(t)\right)^{\mu_i(\lambda)}\\
&= \sum_{\lambda= (n_1, \ldots, n_k) \vdash d}  \frac{d!}{\prod_i (i!)^{\mu_i(\lambda)}}\prod_i\frac{1}{\mu_i(\lambda)!} \left(i!\frac{(-t)^{i}}{t^2(1-t^2)}  \binom {-t}{i}\right)^{\mu_i(\lambda)} \cdot (1-t^2)^{\mu_1(\lambda)}
\\
&= d!\sum_{\lambda= (n_1, \ldots, n_k) \vdash d}  \prod_i\frac{1}{\mu_i(\lambda)!} \left(\frac{(-t)^{i}}{t^2(1-t^2)}  \binom {-t}{i}\right)^{\mu_i(\lambda)} \cdot (1-t^2)^{\mu_1(\lambda)}.
\end{align*}
The result is the factorial $d!$ times a sum of products of terms, each of which depends only on the integers $i$ and $\mu_i(\lambda)$, in the exponential-looking way. $\bigGplus_d(t)$ is expressed as a sum over partitions of $d$ in a way that is similar to the coefficients in certain exponential generating series. A comparison with the exponential formula from enumerative combinatorics \cite[Theorem 3.11]{generatingfunctionology} suggests using the exponential generating series for the sequence of formal series $\bigGplus_d$:
\begin{align*}
    \sum_{d \geq 0}\frac{1}{d!}\bigGplus_d(t)x^d &= \sum_{d \geq 0}\sum_{\lambda= (n_1, \ldots, n_k) \vdash d}  \prod_i\frac{1}{\mu_i(\lambda)!} \left(\frac{(-t)^{i}}{t^2(1-t^2)}  \binom {-t}{i}\right)^{\mu_i(\lambda)} \cdot (1-t^2)^{\mu_1(\lambda)} \cdot x^{\sum_i i \mu_i(\lambda)}\\
    &= \sum_{d \geq 0}\sum_{\lambda= (n_1, \ldots, n_k) \vdash d}  \prod_i\frac{1}{\mu_i(\lambda)!} \left(\frac{(-tx)^{i}}{t^2(1-t^2)}  \binom {-t}{i}\right)^{\mu_i(\lambda)} \cdot (1-t^2)^{\mu_1(\lambda)} 
    \\
    &= \sum_{m_1,m_2,\ldots \in \N_0}  \prod_i\frac{1}{m_i!} \left(\frac{(-tx)^{i}}{t^2(1-t^2)}  \binom {-t}{i}\right)^{m_i} \cdot (1-t^2)^{m_1}
    \\
    &= \prod_{i \geq 2} \sum_{m_i\in \N_0} \frac{1}{m_i!} \left(\frac{(-tx)^{i}}{t^2(1-t^2)}  \binom {-t}{i}\right)^{m_i} \cdot \sum_{m_1 \in \N_0} \frac{1}{m_1!} x^{m_1}\\
    &= \prod_{i \geq 2} \exp \left(\frac{(-tx)^{i}}{t^2(1-t^2)}  \binom {-t}{i}\right) \cdot \exp(x)\\
    &= \exp \left(x+\sum_{i \geq 2}\frac{(-tx)^{i}}{t^2(1-t^2)}  \binom {-t}{i}\right)\\
    &= \exp \left(x+\frac{1}{t^2(1-t^2)}\left(\sum_{i \geq 0}(-tx)^{i} \binom {-t}{i} -1 - t^2x\right)\right)\\
    &= \exp \left(\frac{1}{t^2(1-t^2)}\left((1-tx)^{-t} -1 - t^4x\right)\right),
\end{align*}
as announced.
\end{proof}

Let us resume our computations for small values of $d$. For the full part, the formula from the proof yields
\[\bigGfull_2(t) = 2!\frac{(-t)^{2}}{t^2(1-t^2)} \frac{(-t)(-t-1)}{2!} = \frac{t}{1-t},\]
hence
\[\bigGplus_2 (t) = 1 + \bigGfull_2(t)= \frac{1}{1-t} = \GS_{\N_0}(t),\]
as expected. Next,
\[\bigGfull_3(t) = 3!\frac{(-t)^{3}}{t^2(1-t^2)} \frac{(-t)(-t-1)(-t-2)}{3!} = \frac{t^2(t+2)}{1-t} = 2t^2+3 \sum_{n \geq 3} t^n.\]
Note that, due to the semigroup injection from Theorem \ref{T:FullInjects}, we always have 
\[\bigGfull_d(t) = \frac{d!}{2} \cdot \frac{1}{1-t} + \text{ a polynomial of degree $(2d-4)$ }.\]
Next, from our structural analysis of $\As^+(\T_3)$ it follows that
\begin{align*}
    \bigGplus_3 (t) & = 1 + 3\bigGfull_2(t)+\bigGfull_3(t) = 1 + 3\frac{t}{1-t}+\frac{t^2(t+2)}{1-t} = \frac{(t+1)(t^2+t+1)}{1-t}\\
    & =1 + 3t + 5t^2+6 \sum_{n \geq 3} t^n.
\end{align*}
Similarly,
\[\bigGfull_4(t) = \frac{t^3(t+2)(t+3)}{1-t} = 6t^3+11t^4+12 \sum_{n \geq 5} t^n,\]
and 
\begin{align*}
    \bigGplus_4 (t) & = 1 + 6\bigGfull_2(t) + 4\bigGfull_3(t) + \bigGfull_4(t) + 3\bigGfull_2(t)^2 \\
    &= 1 + 6\frac{t}{1-t} + 4\frac{t^2(t+2)}{1-t} + \frac{t^3(t+2)(t+3)}{1-t} + 3 \left( \frac{t}{1-t} \right)^2.
\end{align*}
Observe that, because of the last summand, this series does not longer stabilise for the degree $n$ large enough.

\section{Normal forms in the structure monoid of the infinite reflection solution} \label{S:MonoidsInfiniteReflection} 

In this section, we prepare the soil for investigating the growth of the structure monoids of reflection solutions $(\R_d, r_{\qu})$. An important ingredient for this will be a normal form in these monoids. This form will be convenient to establish first in a common lift of these solutions (for all $d$), which we will now describe.

\begin{defn}
    The \emph{infinite reflection solution}, denoted by $\R_{\infty}$, is the set $\Z$ endowed with the following map on $\Z \times \Z$:
    \begin{equation}\label{E:ReflectionSol}
     r(x,y) = (x \qu y = -y+2x,x).
    \end{equation}
\end{defn}
One easily checks that it is an invertible non-degenerate YBE solution. Readers familiar with quandles will readily recognise the infinite dihedral quandle.

Taking a quotient modulo $d$, one obtains a solution isomorphic to $(\R_d, r_{\qu})$. Indeed, one just needs to order cyclically the vertices of a regular $d$-gon, and identify the reflection with respect to the vertex $i$ with $i \in \Z_d$.

To compare words in $\As^+(\R_{\infty})$, we will need several natural invariants.

\begin{lem}
    The assignment
    \begin{align*}
        l_0 \times l_1 \colon \As^+(\R_{\infty}) &\to \N_0 \times \N_0, \\
e_{a}  &\mapsto \begin{cases} (1,0) & \text{ if } a \text{ is even},\\
    (0,1) & \text{ if } a \text{ is odd}
\end{cases}
    \end{align*}
    uniquely extends to a well-defined monoid morphism.
\end{lem}
The components $l_0$ and $l_1$ simply count the generators with even and odd indices respectively. More conceptually, $l_0 \times l_1$ is the monoid surjection $\As^+(\R_{\infty}) \twoheadrightarrow \As^+(\R_{2}) \cong \N_0 \times \N_0$ induced by the above mentioned solution surjection $\R_{\infty} \twoheadrightarrow \R_2$.

\begin{proof}
The compatibility with the defining relations in $\As^+(\R_{\infty})$ follows from the fact that $y$ and $-y+2x$ are always of the same parity.
\end{proof}

\begin{defn}
    In what follows, an application of a defining relation in $\As^+(\R_{\infty})$ or its inverse in a rewriting sequence will be called a \emph{braiding move}.
\end{defn}

\begin{defn}
 The numbers $l_0(w)$ and $l_1(w)$ are called the \emph{even} and the \emph{odd length} of $w \in \As^+(\R_{\infty})$ respectively.   
\end{defn}
Their sum is the familiar length $l(w)$ with respect to the generators $e_a$. 

\begin{lem}
    The map
    \begin{align*}
        \w \colon \As^+(\R_{\infty}) &\to \Z, \\
e_{a_1} e_{a_2} \cdots e_{a_n} &\mapsto a_1 - a_2 + \cdots +(-1)^{n+1} a_n,\\
1 &\mapsto 0
    \end{align*}
    is well defined.
\end{lem}

\begin{proof}
This definition is compatible with the braiding moves, since
    \[x-y = (-y+2x)-x. \qedhere\] 
\end{proof}

\begin{defn}
 The number $\w(w)$ is called the \emph{weight} of $w \in \As^+(\R_{\infty})$.   
\end{defn}

This weight function is extended to general Alexander quandles in \cite{Adrien}. Note that $\w$ is not a morphism of monoids, but rather a monoid $1$-cocycle. The weight function is useful for computations in $\As^+(\R_{\infty})$:

\begin{lem}
    For any $w \in \As^+(\R_{\infty})$ and $a \in \Z$, we have
    \begin{equation}\label{E:QuasiCommInAsR}
        we_a = e_{(-1)^{l(w)}a + 2 \w(w)}w.
    \end{equation}
\end{lem}
\begin{proof}
    An easy induction on the length $l(w)$ of $w$.
\end{proof}

A direct consequence of the above formula is the centrality of squares in $\As^+(\R_{\infty})$: 
\begin{lem}
    For any $a,b \in \Z$, we have
\begin{equation}\label{eq:squares_are_central_over_takz}
        e_a^2e_b = e_be_a^2.
    \end{equation}
\end{lem}

\begin{lem}
    The map
    \begin{align*}
        \deg \colon \As^+_{\geq 1}(\R_{\infty}) &\to \N_0, \\
e_{a_1} e_{a_2} \cdots e_{a_n} &\mapsto \gcd(a_1-a_2,a_2-a_3,\ldots,a_{n-1}-a_n), \ n \geq 2,\\
e_a &\mapsto 0,
    \end{align*}
    is well defined.
\end{lem}
Here $\As^+_{\geq 1}(\R_{\infty})$ is the set of elements in $\As^+(\R_{\infty})$ of length $\geq 1$.

\begin{proof}
The $\gcd$ from the above definition can be replaced with the $\gcd$ of all the differences $a_i - a_j$, where $1 \leq i < j \leq n$. To show its stability under braiding moves, use $x-y = (-y+2x)-x$ and $a-(-y+2x) = (a-y)-2(x-y)$.
\end{proof}

\begin{defn}
 The number $\deg(w)$ is called the \emph{density} of $w \in \As^+_{\geq 1}(\R_{\infty})$.   
\end{defn}

The density vanishes if and only if we have $a_1=a_2=\ldots = a_n$. This happens precisely for the powers $e_a^n$. They are called \emph{frozen} since they cannot be rewritten in another way: no non-trivial braiding moves apply to them. In what follows, we will mainly consider the interesting case $\deg(w)>0$.

\begin{lem}
For $w \in \As^+_{\geq 2}(\R_{\infty})$ with $\deg(w)>0$, let $\alpha(w) \in \N_0$ be the minimal representative of the index $a$ of any of its letters $e_a$ modulo $\deg(w)$. For $w = e_a^n \in \As^+_{\geq 1}(\R_{\infty})$ with $\deg(w)=0$, put $\alpha(w)=a$. This yields a well-defined map
    \begin{align*}
        \alpha \colon \As^+_{\geq 1}(\R_{\infty}) &\to \Z.
    \end{align*}
\end{lem}

\begin{proof}
By the definition of $\deg(w)$, all letters from any word representing $w$ give the same residue modulo $d$.
\end{proof}

\begin{defn}
 The number $\alpha(w)$ is called the \emph{anchor} of $w \in \As^+_{\geq 1}(\R_{\infty})$.   
\end{defn}

Note that for any word $1 \neq w = e_{a_1}\ldots e_{a_n}$, the coset $\alpha(w) + \deg(w)\cdot \Z$ is minimal amongst all cosets in $\Z$ that contain the integers $a_1,\ldots,a_n$.

To better understand how these multiple invariants work, the reader might look at their values for different elements:
\begin{center}
\begin{tabular}{>$c<$||*{4}{>$c<$|}>$c<$}
     w & l_0(w) & l_1(w) & \w(w) & \deg(w) & \alpha(w)\\ \hline
     e_1e_1 & 0 & 2 & 0 & 0 & 1 \\ \hline
     e_0e_{-1} & 1 & 1 & 1 & 1 & 0 \\ \hline
     e_1e_{-1} & 0 & 2 & 2 & 2 & 1 \\ \hline
     e_1e_{-1}e_1 & 0 & 3 & 3 & 2 & 1 \\ \hline
     e_{-6}e_{-2}e_{-2} & 3 & 0 & -6 & 4 & 2 
\end{tabular}
\end{center}

To get a complete family of invariants, we need to replace $l_0$ and $l_1$ with something finer.

For integers $d,a$, we denote by $\As_{d,a}^+(R_{\infty})$ the set of all $w \in \As^+_(\R_{\infty})$ such that $\deg(w) = d$ and $\alpha(w) = a$.

\begin{lem}
For $w \in \As_{d,a}^+(R_{\infty})$ where $d > 0$ and $0 \leq a < d$, take any of its representatives, and replace each of its letters $e_x$ with $e_{(x-a)/d}$. This yields a well-defined element $\overline{w}$ with $\deg(\overline{w})=1$.
\end{lem}

Note that the index $(x-a)/d$ can alternatively be expressed as $\lfloor \frac{x}{d} \rfloor$, using the floor function.

\begin{proof}
The numbers $(x-a)/d$ are integers due to the definition of $d=\deg(w)$ and $a=\alpha(w)$. The operation is well-defined since the map $f \colon x \mapsto (x-a)/d$ intertwines the solution $r$ from \cref{E:ReflectionSol}, in the sense of $r \circ (f \times f) = (f \times f) \circ r$. Finally, $\deg(\overline{w}) = \deg(w)/ d =1$.
\end{proof}

\begin{defn}
 The element $\overline{w}$ is called the \emph{essentialisation} of $w \in \As^+_{\geq 2}(\R_{\infty})$. The numbers $\overline{l}_0(w) = l_0(\overline{w})$ and $\overline{l}_1(w) = l_1(\overline{w})$ are called the \emph{essential even} and the \emph{odd length} of $w$ respectively. 
\end{defn}

Here is what this looks like in our examples:
\begin{center}
\begin{tabular}{>$c<$||*{2}{>$c<$|}>$c<$}
     w & \overline{w} & \overline{l}_0(w) & \overline{l}_1(w) \\ \hline
     e_1e_{-1} & e_0e_{-1} & 1 & 1  \\ \hline
     e_1e_{-1}e_1 &  e_0e_{-1}e_0 & 2 & 1 \\ \hline
     e_{-6}e_{-2}e_{-2} & e_{-2}e_{-1}e_{-1} & 1 & 2 
\end{tabular}
\end{center}

Given $d=\deg(w)>0$, $a=\alpha(w)$, and $\overline{w}$, one can reconstruct $w$. We will thus restrict our attention to the case $\deg(w)=1$ (corresponding to the essentialisations) in the major part of what follows.

\begin{defn}
    The elements $w \in \As^+_{\geq 2}(\R_{\infty})$ with $\deg(w)=1$ are called \emph{full}. The set of all such elements is called the \emph{full reflection semigroup}, denoted by $\FullR$.
\end{defn}
This object plays a role analogous to the full transposition monoids $\Full_d$ in Section \ref{S:MonoidsTransp}. It is indeed a semigroup, because of the following property.

\begin{lem}
    The product of any full element $w \in  \FullR$ with any $v \in \As^+(\R_{\infty})$ is full. That is, both $wv$ and $vw$ lie in $\FullR$.
\end{lem}

\begin{proof}
    The differences in the $\gcd$ used to compute $\deg(wv)$ or $\deg(vw)$ include those used to compute $\deg(w)$.
\end{proof}

The purpose of the remainder of this section is to prove that the $5$ invariants above completely determine an element from $\As^+(\R_{\infty})$:

\begin{thm}\label{T:FullInfariantsStrMonRefl}
    Two elements coincide in $\As_{\geq 1}^+(\R_{\infty})$ if and only if they have the same values of the invariants $\deg$, $\alpha$, $\w$, $\overline{l}_0$, and $\overline{l}_1$.
\end{thm}

This will also allow us to construct a normal form in this monoid, and to completely determine the algebraic structure of the full part $\FullR$.

Let us start with considering short elements in $\As^+(\R_{\infty})$.

In length $1$, every element admits a unique presentation, $e_a$, and the index $a$ is exactly the weight of the element: $a = \w(e_a)$.

In length $2$, a braiding move can be written as 
\[w = e_ae_b = e_{-b+2a} e_a = e_{a+\w(w)}e_{a}, \]
hence the set of all the words representing $w$ is 
\begin{equation}\label{E:length2}
  \{ e_{a+k\w(w)}e_{a+(k-1)\w(w)} : k \in \Z\}.  
\end{equation}
Such a set is uniquely determined by $\w(w)$ and $\alpha(w)$ (which is $a$ modulo $\w(w)$). Note that in degree $2$ the density $\deg(w)$ is just the absolute value of the weight.

In the first interesting case, that of length $3$, the weight and the density provide a complete invariant:
\begin{lem} \label{lem:three_element_lemma}
    Any length $3$ element $w \in \As^+(\D_{\infty})$ can be presented as $g = e_{\w(w)+\deg(w)}^2e_{\w(w)}$.
\end{lem}
In particular, this presentation tells us that the anchor $\alpha(w)$ is precisely $\w(w)$ modulo $\deg(w)$.

\begin{proof}
    To a length $3$ word $e_ae_be_c$ representing $w$, assign the difference vector $(d_1,d_2) = (a-b,b-c)$. A rewriting step $e_ae_be_c \sim e_{-b+2a}e_{a}e_c$ or $e_ae_be_c \sim e_ae_{-c+2b}e_b$ changes the difference vector to $(d_1,d_2+d_1)$ or $(d_1-d_2,d_2)$ respectively. Performing the Euclidean algorithm on the couple $(d_1,d_2)$, we obtain an expression $e_{a'}e_{b'}e_{c'} \sim e_a e_b e_c$ with difference vector $(0,d)$, where $d = \gcd(d_1,d_2) = \deg(w)$. It follows that $w = e_{a'}^2a_{a'-d}$. Finally, $\w(w) = a'-a'+(a'-d) = a'-d$, yielding $a'=\w(w) +\deg(w)$.
\end{proof}

The length $2$ and length $3$ cases exhibit fundamentally different behaviour. Starting from length $4$, we at last get a uniform normal form:

\begin{thm}\label{T:ReflNormalForm}
    Any full element $w \in \As^+_{\geq 4}(\R_{\infty})$ can be written as
\begin{equation}\label{E:FinalReduction}
    w=e_0^ke_1^le_c, \qquad \text{ where } k,l > 0, \ c \in \Z.
\end{equation}
We emphasize that we allow $c \in \{0,1 \}$ here. This presentation is unique for $w$ with $l_0(w)$ or $l_1(w)$ equal to $1$. Other $w$ admit exactly two such presentations.
\end{thm}

Our proof begins with several lemmata, interesting per se.

First, we will frequently work with the squares $e_{a}^2$ as if they were simple generators $e_{a}$. Let us explain why this is legitimate.

\begin{lem} \label{lem:powers}
    Let $e_{a_1} \cdots e_{a_m}$ and $e_{b_1} \cdots e_{b_m}$ be two words representing the same element $w \in \As^+(\D_{\infty})$. Let $k_1,\ldots,k_m \in \N$ be $m$ positive integers of the same parity. Then the word $e_{a_1} ^{k_1}\cdots e_{a_m}^{k_m}$ can be rewritten in $\As^+(\D_{\infty})$ as $e_{b_1}^{k_{\sigma(1)}} \cdots e_{b_m}^{k_{\sigma(m)}}$. Here $\sigma \in \Sym_m$ is some permutation. 
\end{lem}
This lemma remains valid for all involutive quandles.

\begin{proof}
    The rewriting procedure for $w$ consists in applying the braiding moves $e_x e_y = e_{x \qu y} e_x$. But we have an analogous relation for powers, $e_x^p e_y^q = e_{x \qu y}^q e_x^p$, in two situations:
    \begin{enumerate}
        \item for odd $p$, since
        \[e_x^p e_y = e_x e_y e_x^{p-1} = e_{x \qu y} e_x^p;\]
        we used that for odd $p$, $e_x^{p-1}$ is a product of squares, hence central by \eqref{eq:squares_are_central_over_takz};
        \item for even $p$ and $q$, since
        \[e_x^p e_y^q = e_x^{p-1} e_y^q e_x = e_{x \qu y}^q e_x^p;\]
        we first used the centrality of $e_y^q$ for even $q$, then the previous calculation for odd $p-1$. \qedhere
    \end{enumerate}
\end{proof}

We also need the following arithmetic lemma:

\begin{lem} \label{lem:triple_gcd_lemma}
    Let $a,b,c \in \Z$, with  $a \neq b$. Then there exists an $n \in \N$ such that 
    \[\gcd(a+nc,b+nc) = \gcd(a,b,c).\]
    Moreover, if $a$ and $b$ are of different parity, then there are both even and odd $n$ satisfying this condition.
\end{lem}

\begin{proof}
    The division by $\gcd(a,b,c)$ reduces the lemma to the case $\gcd(a,b,c) = 1$. We will present $\gcd(a+nc,b+nc)$ in the equivalent form $\gcd(a+nc,b-a)$. A prime divisor $p$ of $b-a$ cannot divide $a$ and $c$ simultaneously, as otherwise it would be also a divisor of $b$ and hence of $\gcd(a,b,c) = 1$. Then there exists an $n_p \in \Z$ such that $p$ does not divide $a+n_pc$. Indeed, one can take $n_p=0$ if $p \nmid a$ and $n_p=1$ otherwise. By the Chinese remainder theorem, there exists an $n \in \N$ congruent to $n_p$ for all prime divisors $p$ of $b-a \neq 0$. Then no prime divisor of $b-a$ divides $a+nc$, and $\gcd(a+nc,b-a)=1$.  

    Now, if $a-b$ is odd, the above argument concerns only odd primes $p$. Therefore, when using the Chinese remainder theorem, one can impose any value of $n \mod 2$.
\end{proof}

\begin{proof}[Proof of \cref{T:ReflNormalForm}]
Take any presentation of a $w \in \As_{\geq 4}^+(\D_{\infty})$ with $\deg(w) = 1$, and start simplifying it. A repeated application of \cref{lem:three_element_lemma} to three-letter segments of $w$ yields a presentation 
\[w = e_{b_1}^2 \cdots e_{b_m}^2 h, \qquad b_j \in \Z, \ h \in \As^+(\D_{\infty}) \text{ of length } \leq 2.\] 

Lemmata \ref{lem:powers} and \ref{lem:three_element_lemma} allow us to transform this presentation into 
\begin{equation} \label{E:partialReduction}
 w = e_a^{2k}e_b^{2l} h, \qquad a,b \in \Z, \ k,l \in \N_0, \ h \in \As^+(\D_{\infty}) \text{ of length } \leq 2.   
\end{equation}

\textbf{Case 1:} a long tail $h=e_ce_d$. 
       
\begin{lem} \label{pro:three_letter_reduction}
    Any $g = e_a^2 e_b e_c \in \FullR$ can be rewritten in at least one of the following forms: 
    \[g = e_0^2e_1e_{1-\w(g)}, \text{ or } g = e_0e_1^2e_{-\w(g)}.\]
\end{lem}

\begin{proof}
The element $e_a e_b e_c$ is still full. By \cref{lem:three_element_lemma}, we can rewrite $g$ as     \begin{equation*}
        g= e_a \cdot e_ae_be_c \sim e_a e_d^2 e_{d-1} \sim e_d^2 e_ae_{d-1}.
    \end{equation*}
    Since rewriting leaves the weight unchanged, we have $\w(g) = a-(d-1)$, so $d = a-\w(g)+1$, and the above rewriting sequence becomes    
    \begin{equation}\label{E:RewritingInfDihedral}
        g \sim e_{a-\w(g)+1}^2 e_a e_{a-\w(g)}.
    \end{equation}
    Observe that the last word depends only on the first letter $e_a$ and the weight $\w(g)$ of the original word. By \cref{eq:squares_are_central_over_takz} and \cref{E:QuasiCommInAsR}, it can be rewritten as
\[e_{a-\w(g)+1}^2 e_a e_{a-\w(g)} \sim e_a e_{a-\w(g)} e_{a-\w(g)+1}^2 \sim e_{a-\w(g)+1 + 2(a-(a-\w(g)))}^2 e_a e_{a-\w(g)} = e_{a+\w(g)+1}^2 e_a e_{a-\w(g)}.\]
Now, applying \cref{E:RewritingInfDihedral} with the first letter $a+\w(g)+1$ instead of $a$ and the same weight $\w(g)$, we obtain
    \begin{equation*}
        g \sim e_{a+2}^2 e_{a+\w(g)+1} e_{a+1}.
    \end{equation*}
A repetition of this procedure or its inverse leaves us either with an expression of the form 
\[e_2^2 e_{\w(g)+1} e_{1} \sim e_2^2e_1 e_{1-\w(g)} \sim e_1e_0^2 e_{1-\w(g)} \sim e_0^2e_1 e_{1-\w(g)},\] 
or with an expression of the form
\[e_{1}^2e_{\w(g)}e_0 \sim e_{1}^2e_0e_{-\w(g)} \sim e_0 e_1^2e_{-\w(g)}. \qedhere\]
\end{proof}

We will now apply a similar reduction to elements of the form $g = e_a^{2k}e_b^{2l}e_ce_d$. 

 First assume $k,l > 0$ and $a \neq b$. \cref{lem:powers} allows us to apply to $e_a^{2k}e_b^{2l}$ the arguments used for the length $2$ case, and rewrite it as $e_{a+n(a-b)}^{2k'}e_{a+(n-1)(a-b)}^{2l'} = e_{a+n(a-b)}^{2k'}e_{b+n(a-b)}^{2l'}$ for any $n \in \Z$ (cf. \cref{E:length2}), with $\{k',l'\} = \{k,l\}$. Putting $a'=a+n(a-b)$ and $b'=b+n(a-b)$, we get
\begin{equation} \label{E:partialReductionCPreontinued}
w=e_{a'}^{2k'}e_{b'}^{2(l'-1)}e_{b'}^{2}e_ce_d.
\end{equation}
We would like to apply \cref{pro:three_letter_reduction} to the part $e_{b'}^{2}e_ce_d$ of this word. For this we have to ensure that it is full. We have
\[\deg(e_{b'}^{2}e_ce_d) = \gcd(b'-c,c-d) = \gcd(b'-c,b'-d) = \gcd(b+n(a-b)-c,b+n(a-b)-d).\]
By \cref{lem:triple_gcd_lemma}, for some $n$ the latter expression equals
\[\gcd(b-c,b-d,a-b) = \gcd(a-b,b-c,c-d) = \deg(e_a^{2k}e_b^{2l}e_ce_d) = 1.\]
We will work with such an $n$, and write again $a$,  $b$, $k$, $l$ instead of $a'$, $b'$, $k'$, $l'$  for simplicity. \cref{pro:three_letter_reduction} now yields
\begin{equation} \label{E:partialReductionContinued}
w=e_{a}^{2k}e_{b}^{2(l-1)}e_{0}^{2}e_1e_{c'} \text{ or } e_{a}^{2k}e_{b}^{2(l-1)}e_{0}e_1^{2}e_{d'},
\end{equation}
with $c'=1-\w(w)$ or $d'=-\w(w)$.

\begin{lem}\label{L:ActionOf01}
    For any $u \in \As^+(\D_{\infty})$ and $t \in \Z$, we have the equality
    \[e_t^2 u e_0e_1 = e_{\tau}^2 u e_0e_1,\]
    where $\tau=0$ for even $t$ and $\tau=1$ for odd $t$.
\end{lem}

\begin{proof}
    Using the centrality of the square $e_t^2$ and  the quasi-commutativity relation \eqref{E:QuasiCommInAsR}, we get   
    \[e_t^2 u e_0e_1 \sim u e_0e_1 e_t^2 \sim u e_{t+2(0-1)}^2 e_0e_1 = e_{t-2}^2 e_0e_1.\]
    Iterating this procedure or its inverse, one replaces $t$ with $\tau$.
\end{proof}

This lemma allows us to replace $a$ and $b$ in
\cref{E:partialReductionContinued} with $\alpha$ and $\beta \in \{0,1\}$. Moving the central elements $e_{\beta}^{2}$ and $e_{\alpha}^{2}$ inside this word, one gets the desired form \eqref{E:FinalReduction}.

If, say, $k=0$ and $l>0$, then in $w=e_{b}^{2(l-1)}e_{b}^{2}e_ce_d$ we directly have $\deg(e_{b}^{2}e_ce_d) = \deg(e_{b}^{2l}e_ce_d) = 1$, and the above arguments apply. 

The case $k=l=0$ is impossible since we are in length ${\geq 4}$.

If $a=b$, then $e_a^{2k}e_b^{2l} =e_b^{2(k+l)}$, and we are in the previous situation. 

\textbf{Case 2:} an empty tail $h=1$. If, say, $l>0$, then the presentation $w= e_a^{2k}e_b^{2(l-1)}e_be_b$ brings us to the previous case. Alternatively, this case can be dealt with using \cref{lem:powers}.

\textbf{Case 3:} a $1$-letter tail $h=e_c$. In length $5$, we have
\[w=e_a^2e_b^2e_c = e_a \cdot e_ae_b^2e_c,\]
with possibly equal $a$ and $b$. The element $e_ae_b^2e_c$ is still full, thus \cref{pro:three_letter_reduction} yields
\[w \sim e_ae_0^2e_1e_{c'} \text{ or } e_ae_0e_1^2e_{d}.\]

In the first case, we once again apply \cref{pro:three_letter_reduction}, together with the representation \eqref{E:partialReduction}, to the underlined subword to get 
\[w \sim e_0 \underline{e_{-a} e_0e_1e_{c'}} \sim e_0 e_0^ke_1^le_{c''} = e_0^{k+1}e_1^le_{c''}, \qquad k,l >0.\]
Note that the element $e_{-a} e_0e_1e_{c'}$ is indeed full because of the index difference $0-1=-1$.

The second case is similar:
\[w \sim e_ae_0e_1e_1e_{d} \sim e_ae_0e_1e_{-d+2} e_{1} \sim e_0^ke_1^le_{d'}e_{1} \sim e_0^ke_1^{l+1}e_{2-d'}, \qquad k,l >0.\]

In greater length, we have
\[w=e_a^{2k}e_b^{2l}e_c \sim e_a^{2(k-1)}e_b^{2(l-1)}e_a^2e_b^2e_c \sim e_a^{2(k-1)}e_b^{2(l-1)} e_0^{k'}e_1^{l'}e_{c'}, \qquad k',l' >0.\]
We assumed $k,l>0$ (otherwise a single power, say $e_a^{2k}$, should be split into two powers), and used the fullness of the element $e_a^2e_b^2e_c$. \cref{L:ActionOf01} allows one to replace $a$ and $b$ with $\alpha$ and $\beta \in \{0,1\}$, and the centrality of squares to rearrange the result into the desired form.

It remains to discuss whether the presentation \eqref{E:FinalReduction} is unique. For a given parity of $c$, the powers $k$ and $l$ are uniquely determined from $l_0(w)$ and $l_1(w)$; the value of $c$ can then be computed from the weight $\w(w)$. Thus $w$ admits at most $2$ desired presentations, one for each parity of $c$. The assumption $l_0(w)=1$ forces $k$ to be $1$ and $c$ to be odd; thus $w$ admits only $1$ desired presentation. The case $l_1(w)=1$ is analogous.

We will now show that elements $w$ with $l_0(w)>1$ and $l_1(w)>1$ admit $2$ desired presentations. In length $4$, these two constraints imply $l_0(w)=l_1(w)=2$. Then $\w(g)$ is even, and in the proof of \ref{pro:three_letter_reduction}, \cref{E:RewritingInfDihedral} allows us to change the parity of the first letter. Then the algorithm from the proof can present $g$ both as $e_0^2e_1e_{1-\w(g)}$ and $e_0e_1^2e_{-\w(g)}$. 

In general even length, let us go back to the presentation $w=e_a^{2k}e_b^{2l}e_ce_d$ that we have obtained before. If $c$ and $d$ are of the same parity, then for the element $e_{b'}^{2}e_ce_d$ in \cref{E:partialReductionCPreontinued} to be full, $b'$ needs to be of the opposite parity. Hence $l_0(e_{b'}^2e_ce_d)=l_1(e_{b'}^2e_ce_d)=2$, and the length $4$ case yields for $w$ two presentations of the form \eqref{E:FinalReduction}, with last letters of different parity. If $c$ and $d$ are of different parity, then so are $a$ and $b$ (otherwise $l_0(w)$ or $l_1(w)= 1$). Then \cref{lem:triple_gcd_lemma} produces an odd and an even $n$ satisfying $\gcd(b+n(a-b)-c,b+n(a-b)-d) = \gcd(b-c,b-d,a-b)$. Since $a-b$ is odd, one can thus force any desired parity of $b'=b+n(a-b)$. Our rewriting algorithm will then produce two presentations of the form \eqref{E:FinalReduction}, with the last letters of different parity.

In the odd length, the indices $a$ and $b$ in the presentation $w=e_a^{2k}e_b^{2l}e_c$ are of different parity (otherwise $l_0(w)$ or $l_1(w) \leq 1$). Since the central elements $e_a^{2k}$ and $e_b^{2l}$ can be interchanged in this expression, one may assume that $a$ and $c$ are of the same parity. But then $l_0(e_ae_b^2e_c) = l_1(e_ae_b^2e_c) = 2$, and once again the length $4$ case yields for $w$ two presentations of the form \eqref{E:FinalReduction}, with the last letters of different parity.
\end{proof}

\cref{T:FullInfariantsStrMonRefl} now follows easily:

\begin{proof}[Proof of \cref{T:FullInfariantsStrMonRefl}]
Since the length $\leq 3$ case was settled in \cref{lem:three_element_lemma} and preceding paragraphs, it remains to show how to reconstruct $w \in \As^+_{\geq 4}(\R_{\infty})$ from the evaluations of the five invariants $\deg$, $\alpha$, $\w$, $\overline{l}_0$, and $\overline{l}_1$. First, for the essentialisation $\overline{w}$, we know the values $l_0(\overline{w}) = \overline{l}_0(w)$, $l_1(\overline{w}) = \overline{l}_1(w)$, and 
\[\w(\overline{w}) = \begin{cases} \left( \w(w)-\alpha(w) \right) / \deg(w) &\text{ in odd length},\\\w(w) / \deg(w) &\text{ in even length}. \end{cases}\]
\cref{T:ReflNormalForm} then allows one to reconstruct $\overline{w}$, which, together with the values $\w(w)$ and $\alpha(w)$, gives back $w$.
\end{proof}

For full elements, these numeric invariants are sufficient not only to distinguish the elements of $\FullR$, but also to describe the algebraic structure of this semigroup:

\begin{thm}\label{T:StructureFRS}
 The following map is an injective morphism of semigroups:
 \begin{align*}
 \iota \colon \FullR &\to \Z \rtimes (\N \times \N),\\
 w &\mapsto (\w(w), l_0(w), l_1(w)). 
 \end{align*}
 Here $\N \times \N$ acts on $\Z$ via $(k,l) \cdot m  = (-1)^{k+l}m$. The image of this inclusion is
 \[\Im(\iota) = \{(m, k,l) : k,l \in \N, \ m \in \Z,\ l \equiv m \mod 2 \}.\]
\end{thm}

\begin{proof}
The map $\iota$ preserves products, since
\begin{align*}
\iota(v)\cdot \iota(w) &= (\w(v), l_0(v), l_1(v)) \cdot (\w(w), l_0(w), l_1(w)) \\
&= (\w(v) + (-1)^{l_0(v)+ l_1(v)} \w(w), l_0(v) + l_0(w), l_1(v)+l_1(w))\\
& = (\w(v) + (-1)^{l(v)} \w(w), l_0(vw), l_1(vw))= (\w(vw), l_0(vw), l_1(vw)) = \iota(vw).
\end{align*}
\cref{T:FullInfariantsStrMonRefl} guarantees its injectivity (recall that for a full $w$, the evaluations $\deg(w)=1$ and $\alpha(w)=0$ are irrelevant). Further, $\w(w)$ and $l_1(w)$ are of the same parity, since the $l_0(w)$ even-indexed letters in $w$ do not change the parity of the weight $\w(w)$, whereas the $l_1(w)$ odd-indexed letters do. Finally, take $k,l \in \N$ and $m \in \Z$ with $l \equiv m \mod 2$. If $l>1$, then
\[(m, k,l) = \iota(e_0^ke_1^{l-1}e_c), \qquad \text{ where } c= \begin{cases} 1-m & k, l \text{ even},\\1+m & l \text{ even}, k \text{ odd},\\m & k \text{ even}, l \text{ odd},\\-m & k, l \text{ odd}. \end{cases}
\]
Indeed, in all four cases $c$ is odd, implying the desired values of $l_0$ and $l_1$, and the weight of $e_0^ke_1^{l-1}e_c$ is $m$. In the same way,
\[(m, k,1) = \iota(e_0^{k-1}e_1e_{1+(-1)^km}).\qedhere\]
\end{proof}

As a consequence, the full part of the structure \textbf{monoid} of the infinite reflection solution injects into the structure \textbf{group} of this solution, exactly as it happened for transposition solutions: 

\begin{cor}\label{C:MonoidIntoGroup}
    The tautological map
    \begin{align*}
        \FullR &\to \As(\R_{\infty}),\\
        e_{a_1}\ldots e_{a_k} &\mapsto e_{a_1}\ldots e_{a_k}
    \end{align*}
    is an injective morphism of semigroups.
\end{cor}

\begin{proof}
It suffices to show that the invariants $\w$, $l_0$ and $l_1$ extend to the group $\As(\R_{\infty})$. To do this, one easily checks that the assignment
 \begin{align*}
 \As(\R_{\infty}) &\to \Z \rtimes (\Z \times \Z),\\
 e_a &\mapsto (a, 1, 0) \text{ for even }a,\\ 
 e_a &\mapsto (a, 0, 1) \text{ for odd }a,\\ 
 \end{align*}
 extends to the entire group $\As(\R_{\infty})$, and that for a product of generators (without any inverses) this map is precisely $\w \times l_0 \times l_1$.
\end{proof}

The algebraic structure of the entire monoid $\As^+(\R_{\infty})$ is much more subtle. For instance, the density of a product depends not only on the density of each component, but also on their anchors:

\begin{lem}\label{L:DensityProduct}
    The density of the product of two elements $v,w \in \As_{\geq 1}^+(\R_{\infty})$ can be computed as follows:
    \begin{equation}\label{E:DensityProduct}
 \deg(vw) = \gcd \left(\deg(v), \deg(w), \alpha(v) - \alpha(w) \right).   
\end{equation}
\end{lem}

\begin{proof}
Putting $v=e_{a_1} e_{a_2} \cdots e_{a_n}$ and $w=e_{b_1} e_{b_2} \cdots e_{b_m}$, we get
\begin{align*}
    \deg(vw) &= \gcd(a_1-a_2,a_2-a_3,\ldots,a_{n-1}-a_n, a_n - b_1, b_1-b_2,\ldots,b_{m-1}-b_m)\\
    &= \gcd(\deg(v), a_n - b_1,\deg(w))= \gcd(\deg(v), \alpha(v) - \alpha(w),\deg(w)).
\end{align*}
Indeed, since $a_n = \alpha(v)+\deg(v)a'$ and $b_1 = \alpha(w)+ \deg(w)b'$ for some $a',b' \in \Z$, we have
\[a_n - b_1 \equiv \alpha(v) - \alpha(w) \mod  \gcd(\deg(v), \deg(w)).\qedhere\]
\end{proof}

\section{How structure monoids of reflection solutions grow} \label{S:MonoidsReflection}

In this section, we will present a normal form for the structure monoids of all finite reflection solutions $(\R_d, r_{\qu})$, inherited from the normal form for the infinite reflection solution $(\R_{\infty}, r)$ (\cref{T:ReflNormalForm}). This will allow us to compute the growth series of these monoids. Recall that we identify the set $\R_d$ with $\Z_d$; the solution $r_{\qu}$ then takes the form \eqref{E:ReflectionSol}. 

We will start by adapting the five invariants for $\As_{\geq 1}^+(\R_{\infty})$ to the finite setting. The proofs of the following results repeat verbatim those of the analogous statements from \cref{S:MonoidsInfiniteReflection} ; one simply turns all equalities into congruences modulo $d$.

\begin{lem}
    The map
    \begin{align*}
        \deg \colon \As^+_{\geq 1}(\R_{d}) &\to \Z, \\
e_{a_1} e_{a_2} \cdots e_{a_n} &\mapsto \gcd(d,a_1-a_2,a_2-a_3,\ldots,a_{n-1}-a_n), \ n \geq 2,\\
e_a &\mapsto d,
    \end{align*}
    is well defined.
\end{lem}

Note that here it is convenient to impose $\deg(e_a)=d$ rather than $0$, which was our choice in the infinite case.

\begin{defn}
 The number $\deg(w) \in \Z$ is called the \emph{density} of $w \in \As^+_{\geq 1}(\R_d)$. The elements $w$ with $\deg(w)=1$ are called \emph{full}. The set of all such elements is called the \emph{full reflection semigroup of level $d$}, denoted by $\FullR_d$.
\end{defn}

We will soon see that $\FullR_d$ is indeed a sub-semigroup of $\As^+_{\geq 1}(\R_d)$.

\begin{lem}
For $w \in \As^+_{\geq 1}(\R_{d})$, let $\alpha(w) \in \N_0$ be the minimal representative of the index $a$ of any of its letters $e_a$ modulo $\deg(w)$. This yields a well-defined map
    \begin{align*}
        \alpha \colon \As^+_{\geq 1}(\R_{d}) &\to \N_0.
    \end{align*}
\end{lem}

\begin{defn}
 The number $\alpha(w)$ is called the \emph{anchor} of $w \in \As^+_{\geq 1}(\R_{d})$.   
\end{defn}

\begin{nota}
    The (possibly empty) set of all $w \in \As^+_{\geq 1}(\R_{d})$ with the same density $c$ and the same anchor $a$ is denoted by $\As^+_{c,a}(\R_{d})$.
\end{nota}

This set is non-empty if and only if $c$ is a divisor of $d$ and $a$ lies in $\{0,1,\ldots, c-1\}$.

\begin{pro}\label{P:ReductionToFull}
    For any $c \mid d$ and $a \in \{0,1,\ldots, c-1\}$, the set $\As^+_{c,a}(\R_{d})$ is a sub-semigroup of $\As^+_{\geq 1}(\R_{d})$ isomorphic to the full reflection semigroup $\FullR_{\frac{d}{c}}$:
    \[\As^+_{c,a}(\R_{d}) \cong \FullR_{\frac{d}{c}}.\]
\end{pro} 

\begin{proof}
Relation \eqref{E:DensityProduct} for densities is still valid in $\As^+_{\geq 1}(\R_{d})$. For $v$ and $w$ sharing the same density $c$ and the same anchor $a$, it yields
\[\deg(vw) = \gcd \left(c, c, a-a \right) = c.\]   
Also, the indices of all letters in any expression for $vw$ are congruent to $a \mod c$. Hence $vw \in \As^+_{c,a}(\R_{d})$, and the latter is a sub-semigroup.

Next, consider the arithmetic bijection
\begin{align*}
f \colon \{ x \in \Z_d : x \equiv a \mod c \} &\overset{\sim}{\to} \Z_{\frac{d}{c}},\\ 
    x &\mapsto \frac{x-a}{c}.
\end{align*}
Like any affine function, it intertwines the solution $r$ from \cref{E:ReflectionSol}, and thus induces a semigroup injection $\overline{f} \colon \As^+_{c,a}(\R_{d}) \hookrightarrow \As^+_{\geq 1}(R_{\frac{d}{c}})$. Moreover, the condition $\deg(w)=c$ implies $\deg(\overline{f}(w))=1$, thus $\overline{f}$ actually injects $\As^+_{c,a}(\R_{d})$ into $\FullR_{\frac{d}{c}}$. To see the surjectivity, for any $u = e_{a_1} \cdots e_{a_n} \in \FullR_{\frac{d}{c}}$, consider the well-defined element $u' = e_{a_1c+a} \cdots e_{a_nc+a} \in \As^+(\R_{d})$. Its degree is $c \cdot \deg(u) = c$, and its anchor is $a$; hence $u' \in \As^+_{c,a}(\R_{d})$. By construction, $\overline{f}(u')=u$, thus $u'$ is a preimage of $u$. Hence $\overline{f}$ yields the desired semigroup isomorphism.
\end{proof}

Thus the semigroup $\As^+_{\geq 1}(\R_{d})$ splits into a disjoint union of sub-semigroups. For each divisor $c$ of $d$, there are exactly $c$ sub-semigroups isomorphic to $\FullR_{\frac{d}{c}}$. The computation of the growth series of the structure monoids $\As^+(\R_{d})$ is thus reduced to that of its full components. Note that as an abuse of notation, we write
\[
\GS_{\FullR_d}(t) = \sum_{g \in \FullR_d} t^{l(g)},
\]
which is not a growth series but a restricted growth series.

\begin{cor}
For any integer $d \geq 2$, the growth series of the monoids $\As^+(\R_{d})$ and $\FullR_{c}$ are related as follows:
    \begin{equation}\label{E:GrowthReflMonoidReducedToFull}
        \GS_{\As^+(\R_d)} (t) = 1 + \sum_{c|d} \frac{d}{c} \cdot \GS_{\FullR_c}(t).
    \end{equation}
\end{cor}
Note that we exchanged the roles of $c$ and $\frac{d}{c}$ from the preceding arguments. 

We will thus focus on the full semigroups $\FullR_{d}$. Their structure will be entirely captured by the weight and the length functions, the latter being split into the odd and the even parts for even $d$. 

\begin{lem}
    For any integer $d \geq 2$, the assignment
    \begin{align*}
        l \colon \As^+(\R_{d}) &\to \N_0, \\
e_{a}  &\mapsto 1
    \end{align*}
    uniquely extends to a well-defined semigroup morphism. For even $d$, is can be refined into another semigroup morphism by imposing    
    \begin{align*}
        l_0 \times l_1 \colon \As^+(\R_d) &\to \N_0 \times \N_0, \\
e_{a}  &\mapsto \begin{cases} (1,0) & \text{ if } a \text{ is even},\\
    (0,1) & \text{ if } a \text{ is odd}.
\end{cases}
    \end{align*}
    For even $d$, we consider an element in $\Z_d$ to be even resp. odd if any of its representatives in $\Z$ is even resp. odd. 
\end{lem}

\begin{defn}
 The numbers $l(w)$, $l_0(w)$, and $l_1(w)$ are called the \emph{length}, the \emph{even length}, and the \emph{odd length} of $w$ respectively.   
\end{defn}

\begin{lem}
    The map
    \begin{align*}
        \w \colon \As^+(\R_{d}) &\to \Z_d, \\
e_{a_1} e_{a_2} \cdots e_{a_n} &\mapsto a_1 - a_2 + \cdots +(-1)^{n+1} a_n,\\
1 &\mapsto 0
    \end{align*}
    is well defined.
\end{lem}

\begin{defn}
 The element $\w(w) \in \Z_d$ is called the \emph{weight} of $w$.   
\end{defn}

Now, we are ready to adapt the normal form from \cref{T:ReflNormalForm} to the finite setting:

\begin{thm}\label{T:ReflNormalFormFinite}
    Any full element $w \in \As^+_{\geq 4}(\R_{d})$ can be written as
\begin{equation}\label{E:FinalReductionFinite}
    w=e_0^ke_1^le_c, \qquad \text{ where } k,l > 0, \ c \in \Z_d.
\end{equation}
For even $d$, this presentation is unique for $w$ with $l_0(w)$ or $l_1(w)$ equal to $1$, whereas other $w$ admit exactly two such presentations. 

For odd $d$ and any $w$, there is exactly one such presentation with $l=1$.
\end{thm}

Before presenting the proof, we need a generalisation of \cref{lem:triple_gcd_lemma}:

\begin{lem} \label{lem:lifting_to_coprime_elements}
Take $k \geq 2$ integers $a_1, \ldots,a_k$, and some $d \in \Z$. Then there are integers $m_1,\ldots,m_k$ such that
    \[
    \gcd(a_1+m_1d, \ldots, a_k+m_kd) = \gcd(d,a_1,a_2,\ldots,a_k).
    \]
    If $d$ is odd, then these integers can be chosen such that $a_i + m_id$ is odd for all $1 \leq i \leq k$. 
\end{lem}

\begin{proof}
    Without loss of generality, we may assume $a_1 \neq 0$. Indeed, for this one should permute the $a_i$ if necessary; this does not work only when all $a_i = 0$, in which case the statement is trivial.
    
    We first prove that for all $a_1,a_2,d \in \Z$, there is an $m$ such that $\gcd(a_1,a_2+md) = \gcd(d,a_1,a_2)$. Dividing by $\gcd(d,a_1,a_2)$, we may assume that $\gcd(d,a_1,a_2) = 1$. For any prime divisor $p | a_1$, there exists an integer $m_{(p)} \in \Z$ such that $p$ does not divide $a_2 + md$ for any $m \equiv m_{(p)}\mod p$. Applying the Chinese remainder theorem, we obtain an $m \in \Z$ such that $a_2 + md$ and $a_1$ have no common prime divisors, providing $\gcd(a_1,a_2+md) = 1$.

    Now for the general case, let $d,a_1,\ldots,a_k \in \Z$. The previous paragraph provides us with $m_2,\ldots, m_k \in \Z$ such that $\gcd(a_1,a_i+m_id) = \gcd(d,a_1,a_i)$ for $1 \leq i \leq k$. Then with $m_1 = 0$, we have
    \begin{align*}
    \gcd(d,a_1,\ldots,a_k) & = \gcd \left( \gcd(d,a_1,a_2),\gcd(d,a_1,a_3),\ldots,\gcd(d,a_1,a_k) \right) \\ &
    = \gcd \left( \gcd(a_1,a_2+m_2d),\gcd(a_1,a_3+m_3d),\ldots,\gcd(a_1,a_k+m_kd) \right) \\ 
    & =  \gcd(a_1,a_2+m_2d,\ldots a_k+m_kd) = \gcd(a_1+m_1d,\ldots,a_k+m_kd).
    \end{align*}

    If $d$ is odd, then replace $a_i$ with $a_i' = a_i+m_i'd$ in a way such that $a_i'$ is odd for all $1 \leq i \leq k$. Then
    \[
    \gcd(d,a_1,\ldots,a_k) = \gcd(d,a_1',\ldots,a_k') = \gcd(2d,a_1',\ldots,a_k').
    \]
    Now the first statement of the lemma shows the existence of $m_1'',\ldots,m_k''$ such that
    \begin{align*}
        \gcd(2d,a_1',\ldots,a_k') &= \gcd(a_1'+ m_1''\cdot 2d,\ldots,a_k'+m_k'' \cdot 2d)\\ 
        &= \gcd(a_1+(m_1'+2m_1'')d,\ldots,a_k+(m_k'+2m_k'')d)\\ 
        &= \gcd(a_1'+2m_1''d,\ldots,a_k'+2m_k''d).
    \end{align*}
    It follows that $\gcd(d,a_1,\ldots,a_k) = \gcd(a_1'+2m_1''d,\ldots,a_k'+2m_k''d)$. Since the $a_i'$ are odd, so are the $a_i'+2m_i''d$.
\end{proof}

\begin{proof}[Proof of \cref{T:ReflNormalFormFinite}]
 Take an element 
 \[u=e_{a_1}\cdots e_{a_k} \in \FullR_d.\] 
 Its fullness can be expressed by the condition
 \[\gcd(d,a_1-a_2, a_1-a_3,\ldots,a_1 - a_k) = 1.\]
 By \cref{lem:lifting_to_coprime_elements}, there are some $m_i \in \Z$ with
  \[\gcd(a_1-a_2+m_1d, a_1-a_3+m_2d,\ldots,a_1 - a_k+m_{k-1}d) = 1.\]
  This allows us to lift $u$ to a full element 
  \[u'=e_{a_1}e_{a_2-m_1d}\cdots e_{a_k-m_{k-1}d} \in \FullR.\]
  By \cref{T:ReflNormalForm}, it can be presented as
\[u'=e_0^ke_1^le_{c'} \qquad \text{ for some } k,l > 0, \ c' \in \Z.\]
 Taking a quotient modulo $d$ yields a new presentation of $u$:
 \[u=e_0^ke_1^le_c \qquad \text{ for some } k,l > 0, \ c \in \Z_d.\]
 
 For even $d$, the parameters $k$, $l$ and $c$ are determined by $\w(u)$, $l_0(u)$, $l_1(u)$ and the parity of $c$. The number of such presentations is then determined in the same way as in the infinite case (\cref{T:ReflNormalForm}).

 For odd $d$, we assume $a_1$ odd; otherwise replace each $a_i$ with $a_i+d$. \cref{lem:lifting_to_coprime_elements} guarantees that all elements $a_1 - a_i+m_{i-1}d$ above can be assumed odd, implying that all the $a_i-m_{i-1}d$ are even. This is translated by $l_1(u')=1$. Then the presentation from \cref{T:ReflNormalForm}, takes the particular form
\[u'=e_0^{k}e_1e_{c'},\]
inducing an analogous presentation for $u$. Here $k=l(u)-2$, and the parameter $c$ for $u$ is then uniquely determined by $\w(u)$. Such a presentation is thus unique.
\end{proof}

In a similar way, one adapts \cref{T:StructureFRS}:
\begin{thm}\label{T:StructureFRSFinite}
 For even $d$, the following map is an injective morphism of semigroups:
 \begin{align*}
 \iota_d \colon \FullR_d &\to \Z_d \rtimes (\N \times \N),\\
 w &\mapsto (\w(w), l_0(w), l_1(w)). 
 \end{align*}
 The image of this inclusion is
 \[\Im(\iota_d) = \left( \{(m, k,l) : k,l \in \N, k+l>2 \ m \in \Z_d,\ l \equiv m \mod 2 \} \right) \sqcup (\Z_d^* \times \{(1,1)\}).\]
 where $\Z_d^* = \{ a \in \Z_d : \langle a \rangle = \Z_d \}$.
 
 For odd $d>1$, the following map is an injective morphism of semigroups:
 \begin{align*}
 \iota_d \colon \FullR_d &\to \Z_d \rtimes \N_{\geq 2},\\
 w &\mapsto (\w(w), l(w)). 
 \end{align*}
  The image of this inclusion is
 \[\Im(\iota_d) = (\Z_d \times \N_{\geq 3}) \sqcup (\Z_d^* \times \{2\}).\]
 
 For $d=1$, the length function $l$ realises a semigroup isomorphism $\FullR_1 \cong \N$.
\end{thm}
 Here $\N \times \N$ and $\N$ act on $\Z_d$ via $(k,l) \cdot m  = (-1)^{k+l}m$ and $k \cdot m  = (-1)^{k}m$ respectively.

Note that the injectivity statements from the theorem, together with \cref{P:ReductionToFull}, allow one to exhibit a complete list of numerical invariants for the monoid $\As^+(\R_{d})$ ($5$ for even $d$ and $4$ for odd $d$), in the spirit of \cref{T:FullInfariantsStrMonRefl}. This theorem also implies that the full part of the structure \textbf{monoids} of all finite reflection solutions inject into the corresponding structure \textbf{groups}, as it was established in the infinite case in \cref{C:MonoidIntoGroup}.

\begin{proof}
 The $d=1$ case is straightforward. In the remainder of the proof we assume $d>1$. In particular, all full elements are of length $\geq 2$.
 
 The well-definedness of the maps $\iota_d$ is checked in the same way as for the infinite reflection solution $\R_{\infty}$. Their injectivity in length $\geq 4$ is proved using the normal forms from \cref{T:ReflNormalFormFinite}. 
 
 In length $3$, the arguments from the proof of \cref{T:ReflNormalFormFinite} together with \cref{pro:three_letter_reduction} yield presentations of the form 
 \[w=e_{\w(w)+1}^2e_{\w(w)}.\] 
 Note that an element of this form is necessarily full because of the index difference $(\w(w)+1)-\w(w)=1$. In length $2$, the fullness of $w$ is equivalent to the invertibility of its weight, $\w(w) \in \Z_d^*$. Then the discussion before \cref{pro:three_letter_reduction} yields presentations of the form 
 \[w=e_{0}e_{-\w(w)}.\] 
 In both cases, the weight $\w(w)$ determines $w$ uniquely. Moreover, taking elements of the two particular forms above, one realises any weight in $\Z_d$ and $\Z_d^*$ respectively. This describes the corresponding components of $\Im(\iota_d)$.
 
From now on, we work in length $\geq 4$.
 
Let us prove the surjectivity of $\iota_d$ for odd $d>1$. Given some $m \in \Z_d$ and $k \in \N_{\geq 4}$, the element 
\[w=e_0^{k-2}e_1e_c \in \As^+(\R_{d}),\qquad \text{ where } c=\begin{cases}
     1+m & \text{ for odd } k,\\
     1-m & \text{ for even } k
 \end{cases}\]
 is full because of the index difference $0-1$, and satisfies 
 $l(w) = k$ and $\w(w)=m$. This completes the description of $\Im(\iota_d)$ in this case. 

For even $d$, the proof of the suggested description of $\Im(\iota_d)$ in length $\geq 4$ repeats verbatim the proof of \cref{T:StructureFRS}.
\end{proof}

The realisation of the monoids $\FullR_d$ inside simpler monoids with easily extractable length allows us to easily count the number of their elements of given length, and hence compute the growth series:

\begin{cor} \label{pro:growth_of_fds}
    For any integer $d \in \N$, we have
    \[
    \GS_{\FullR_d}(t) = \begin{cases}
        \frac{t}{1-t}, & d = 1; \\
        \varphi(d)t^2 + d \frac{t^3}{1-t}, & d \textnormal{ odd}, d > 1; \\
        \varphi(d)t^2 + \frac{d}{2}\frac{t^3(2-t)}{(1-t)^2}, & d \textnormal{ even}. 
    \end{cases}
    \]
\end{cor}

Notation $\varphi(n)$ stands here for \emph{Euler's totient function}, which counts the number of integers between $1$ and $n$ coprime with $n$.

\begin{proof}
For $d=1$, \cref{T:StructureFRSFinite} provides exactly one element in each length $\geq 1$. 

For odd $d>1$, it gives $\varphi(d) = \#\Z_d^*$ elements in length $2$, and $d$ elements in each length $\geq 3$.

For even $d$, it gives again $\varphi(d) = \#\Z_d^*$ elements in length $2$, and, in each length $n \geq 3$, the number of elements is
\[\frac{d}{2} \cdot \#\{k,l \in \N : k+l=n\} = \frac{d}{2} \cdot (n-1) .\]
We thus need to compute
\begin{align*}
    \sum_{n \geq 3} (n-1)t^n &= \left( \sum_{n \geq 4} t^n \right)'-2\sum_{n \geq 3} t^n = \frac{t^3(2-t)}{(1-t)^2}.\qedhere
\end{align*}
\end{proof}

To compute the growth series of the entire structure monoids $\As^+(\R_d)$, it remains to plug our formulas for the full parts into \cref{E:GrowthReflMonoidReducedToFull}.  

\begin{thm}\label{T:GrowthStrMonoidReflectionFinite}
For any $d \in \N$, the growth series of the structure monoid of the size $d$ reflection solution has the following form:
    \[\GS_{\As^+(\R_d)} (t)=1+ d \cdot \left( t + \left( \sum_{c|d } \frac{\varphi(c)}{c} \right) t^2 + \tau(d) \frac{t^3}{1-t}  + \tau\left(\frac{d}{2}\right) \frac{t^4}{2(1-t)^2} \right).\]
    Here $\tau$ is the number-of-divisors function:
    \[\tau(a) = \begin{cases}
        \# \{c \in \N : c|a\}, & a \in \N,\\
        0, & a \notin \N.
    \end{cases} \]
\end{thm}

Note that the expression $d \cdot \sum_{c|d } \frac{\varphi(c)}{c}$ can be seen as the convolution of $\varphi$ and $\Id$.

The result depends heavily on the arithmetic properties of $d$. This is not surprising, given that the algebraic structure of dihedral groups depends heavily on the arithmetic properties of their size.

\begin{proof}
    Let us split the sum in \cref{E:GrowthReflMonoidReducedToFull} into three parts, according to the three cases in our computation of $\GS_{\FullR_c}(t)$:
    \begin{align*}
    \GS_{\As^+(\R_d)} (t) &= 1 + d \cdot \GS_{\FullR_1}(t)+ \sum_{\substack{1<c|d \\ c \text{ odd}} }\frac{d}{c} \cdot \GS_{\FullR_c}(t)+ \sum_{\substack{c|d \\ c \text{ even}}} \frac{d}{c}\cdot \GS_{\FullR_c}(t)\\
    &=1+ d \cdot \frac{t}{1-t} +d \cdot \sum_{\substack{1<c|d \\ c \text{ odd}} } \left(\frac{\varphi(c)}{c}t^2 + \frac{t^3}{1-t}\right) +d \cdot \sum_{\substack{c|d \\ c \text{ even}} } \left(\frac{\varphi(c)}{c}t^2 + \frac{t^3(2-t)}{2(1-t)^2}\right)\\
    &=1+ d \cdot \left( \frac{t}{1-t} + \left( \sum_{1<c|d } \frac{\varphi(c)}{c} \right) t^2 + \sum_{1<c|d } \frac{t^3}{1-t}  + \sum_{\substack{c|d \\ c \text{ even}} } \frac{t^4}{2(1-t)^2} \right)\\
    &=1+ d \cdot \left( t + \left( \sum_{c|d } \frac{\varphi(c)}{c} \right) t^2 + \sum_{c|d } \frac{t^3}{1-t}  + \sum_{\substack{c|d \\ c \text{ even}} } \frac{t^4}{2(1-t)^2} \right)\\
    &=1+ d \cdot \left( t + \left( \sum_{c|d } \frac{\varphi(c)}{c} \right) t^2 + \tau(d) \frac{t^3}{1-t}  + \tau\left(\frac{d}{2}\right) \frac{t^4}{2(1-t)^2} \right).\qedhere
    \end{align*}
\end{proof}

\bibliographystyle{abbrv}
\bibliography{refs}

\begin{thebibliography}{10}

\bibitem{SymConj}
E.~Adan-Bante and H.~Verrill.
\newblock Symmetric groups and conjugacy classes.
\newblock {\em J. Group Theory}, 11(3):371--379, 2008.

\bibitem{GrowthRacines}
R.~Bacher, P.~de~la Harpe, and B.~Venkov.
\newblock S\'eries de croissance et polyn\^omes d'{E}hrhart associ\'es aux r\'eseaux de racines.
\newblock {\em Ann. Inst. Fourier (Grenoble)}, 49(3):727--762, 1999.
\newblock Symposium \`a{} la M\'emoire de Fran\c cois Jaeger (Grenoble, 1998).

\bibitem{Bessis_dual_braid_monoid}
D.~Bessis.
\newblock The dual braid monoid.
\newblock {\em Annales Scientifiques de l’École Normale Supérieure}, 36(5):647--683, 2003.

\bibitem{GrowthSurfaceGroups}
J.~W. Cannon and P.~Wagreich.
\newblock Growth functions of surface groups.
\newblock {\em Math. Ann.}, 293(2):239--257, 1992.

\bibitem{StrMon2}
F.~Ced\'o, E.~Jespers, L.~Kubat, A.~Van~Antwerpen, and C.~Verwimp.
\newblock On various types of nilpotency of the structure monoid and group of a set-theoretic solution of the {Y}ang-{B}axter equation.
\newblock {\em J. Pure Appl. Algebra}, 227(2):Paper No. 107194, 38, 2023.

\bibitem{StrMon1}
F.~Ced\'o, E.~Jespers, and C.~Verwimp.
\newblock Structure monoids of set-theoretic solutions of the {Y}ang-{B}axter equation.
\newblock {\em Publ. Mat.}, 65(2):499--528, 2021.

\bibitem{GrowthGraphProduct}
I.~M. Chiswell.
\newblock The growth series of a graph product.
\newblock {\em Bull. London Math. Soc.}, 26(3):268--272, 1994.

\bibitem{YBEGarside}
F.~Chouraqui.
\newblock Garside groups and {Y}ang--{B}axter equation.
\newblock {\em Comm. Algebra}, 38(12):4441--4460, 2010.

\bibitem{Clauwens}
F.~J. B.~J. Clauwens.
\newblock {The adjoint group of an Alexander quandle}, 2010.

\bibitem{Adrien}
A.~Clément.
\newblock {Structure groups and second homology groups of linear Alexander quandles}, 2025.

\bibitem{IglesiasVen}
A.~Garc\'ia~Iglesias and L.~Vendramin.
\newblock An explicit description of the second cohomology group of a quandle.
\newblock {\em Math. Z.}, 286(3-4):1041--1063, 2017.

\bibitem{StrMonAlg2}
T.~Gateva-Ivanova, E.~Jespers, and J.~Okni\'nski.
\newblock Quadratic algebras of skew type and the underlying monoids.
\newblock {\em J. Algebra}, 270(2):635--659, 2003.

\bibitem{GIVdB}
T.~Gateva-Ivanova and M.~Van~den Bergh.
\newblock Semigroups of {$I$}-type.
\newblock {\em J. Algebra}, 206(1):97--112, 1998.

\bibitem{StrMonAlg}
E.~Jespers, L.~u. Kubat, and A.~Van~Antwerpen.
\newblock The structure monoid and algebra of a non-degenerate set-theoretic solution of the {Y}ang-{B}axter equation.
\newblock {\em Trans. Amer. Math. Soc.}, 372(10):7191--7223, 2019.

\bibitem{LebedConj}
V.~Lebed.
\newblock Conjugation groups and structure groups of quandles, 2024.

\bibitem{LebVenStrGroups}
V.~Lebed and L.~Vendramin.
\newblock On structure groups of set-theoretic solutions to the {Y}ang--{B}axter equation.
\newblock {\em Proc. Edinb. Math. Soc. (2)}, 62(3):683--717, 2019.

\bibitem{HowGroupsGrow}
A.~Mann.
\newblock {\em How groups grow}, volume 395 of {\em London Mathematical Society Lecture Note Series}.
\newblock Cambridge University Press, Cambridge, 2012.

\bibitem{Ore_Commutators}
O.~Ore.
\newblock Some remarks on commutators.
\newblock {\em Proceedings of the American Mathematical Society}, 2(2):307--314, 1951.

\bibitem{GrowthDyer}
L.~Paris and O.~Varghese.
\newblock The spherical growth series of {D}yer groups.
\newblock {\em Proc. Edinb. Math. Soc. (2)}, 67(1):168--187, 2024.

\bibitem{ShephardTodd}
G.~C. Shephard and J.~A. Todd.
\newblock Finite unitary reflection groups.
\newblock {\em Canad. J. Math.}, 6:274--304, 1954.

\bibitem{Vendramin_Survey}
L.~Vendramin.
\newblock Skew braces: a brief survey.
\newblock In {\em Geometric methods in physics {XL}}, Trends Math., pages 153--175. Birkh\"auser/Springer, Cham, [2024] \copyright 2024.

\bibitem{generatingfunctionology}
H.~Wilf.
\newblock {\em Generatingfunctionology: {T}hird {E}dition}.
\newblock Ak Peters Series. Taylor \& Francis, 2006.

\end{thebibliography}

\end{document}